\documentclass[11pt]{article}
\usepackage{epsfig}
\usepackage{amsmath,subfigure}
\usepackage{amsthm}
\usepackage{amssymb}
\usepackage{color}
\usepackage[utf8]{inputenc}
\usepackage[english]{babel}

\usepackage[titletoc,toc,title]{appendix}

\usepackage[top=2 cm,bottom=2 cm,left=2. cm, right=2. cm]{geometry}

\usepackage{dsfont}

{\nopagebreak\refstepcounter{theorem} \par{\it\ Assumption}\ \thetheorem{}. \ignorespaces\nopagebreak}%

\newtheorem{thm}{Theorem}
\newtheorem{prop}{Proposition}
\newtheorem{defi}{Definition}
\newtheorem{cor}{Corollary}
\newtheorem{lem}{Lemma}
\newtheorem*{remark}{Remark}

\newcommand\blfootnote[1]{%
	\begingroup
	\renewcommand\thefootnote{}\footnote{#1}%
	\addtocounter{footnote}{-1}%
	\endgroup
}

\title{Chemical Kinetics, Markov Chains, and the Imaginary It\^{o} Interpretation}
\author{\'{A}lvaro Correales, Carlos Escudero, Mariya Ptashnyk} %\thanks{This research was supported by NRP early career research exchanges grant}
\begin{document}
\maketitle

\begin{abstract}
The abstract chemical reaction
$$
A+A \to \emptyset,
$$
understood as a Markov chain in continuous time, has been studied in the physical literature for several years.
It has been claimed that this reaction can be described by means of the stochastic differential equation
$$
d \phi = - \phi^2 dt + i \, \phi \, dW_t,
$$
where $i$ is the imaginary unit. This affirmation is, at least, intriguing, and has led to controversy and criticisms in the literature.
The goal of this work is to give partial evidence that such a description may be possible.
\end{abstract}

\blfootnote{ \hspace{-.7cm} {\it Keywords:} Duality, It\^o diffusions on the complex plane, Markov chains, partial differential equations,
stochastic differential equations, stochastic reaction processes. }

\section{Introduction}

\label{intro}

The abstract chemical reaction
\begin{equation}\label{mainreact}
A + A \stackrel{\lambda}{\longrightarrow} \emptyset,
\end{equation}
denotes in the physical literature a continuous time Markov chain with an infinite state space $\left\lbrace n\right\rbrace$,
\\ $n\in \mathbb{N} \cup \{0\}$,
and which probability distribution is described by the differential equation
\begin{equation}\label{master}
\frac{d}{dt} P_n(t) = \frac{\lambda}{2} [(n+2)(n+1)P_{n+2}(t)-n(n-1)P_n(t)].
\end{equation}
The explicit solution of this forward Kolmogorov or master equation is well known since long ago~\cite{Mcq1,Mcq2}. Despite of this fact, or perhaps
as a consequence of it, research regarding this particular Markov chain has grown since then. One of the most intriguing affirmations
regarding this Markov process is its equivalence to the solution of the stochastic differential equation (SDE)
\begin{equation}\label{isde}
d \phi = - \phi^2 dt + i \, \phi \, dW_t,
\end{equation}
where $i$ is the imaginary unit, after the rescaling of time $t \to t/\lambda$ has been performed. It is clear that at least two facts
can be surprising in this claim; first, we are moving from a discrete space state in~\eqref{master} to a continuous one in~\eqref{isde} while claiming they are both equivalent, and not just an approximation of one another
(despite of the fact that this equation is reminiscent of a continuum limit of the Markov chain, see Appendix~\ref{sectvan}).
Second, a purely jump stochastic process is assimilated to a diffusion in the complex plane. Of course, the precise meaning of the
word \emph{equivalence} in this context will be key in unveiling the potential relations between equations~\eqref{master} and~\eqref{isde},
if any. Let us start summarizing how this idea appears and develops in the literature.
The use of SDEs with an imaginary diffusion to describe Markov chains modeling stoichiometric relations dates back to 1977~\cite{gardiner0}.
Equations of the type of~\eqref{isde} were developed in the context of the Poisson representation, which is connected with earlier quantum
theory~\cite{gardiner}.
The same idea reappeared in~\cite{cardy2}, where this SDE
is derived from equation~\eqref{master} by means of formal but sophisticated field-theoretic
methods\footnote{Actually, a stochastic partial differential equation of reaction-diffusion type is derived.
But our SDE follows from the same argument provided a zero-dimensional system were considered. Also, note
	the difference in the notations used.}. This new formalism, again of quantum-theoretic inspiration, gave an increased popularity
to the use of these equations,
which reappeared in the literature many times since then, like for instance in~\cite{cardy3,cardy4, dfhk,gdl, howard,alex,munoz1,tauber},
where this list is not meant in any sense to be exhaustive.
Despite of this popularity, no rigorous derivations, to the best of our knowledge, are present anywhere;
the same formal field-theoretic or Poisson-representation methods referred to above,
which seem to be equivalent to a large extent~\cite{droz}, are always employed. Although seemingly accepted for years, the description of~\eqref{master} in terms of~\eqref{isde} has been recently put into question in~\cite{benitez} and~\cite{wiese}.
However, this contraposition of simultaneously formal, but sophisticated, field-theoretic
arguments does not clarify what is the precise range of validity of equation~\eqref{isde}, if any. Therefore it seems that a careful stochastic analysis
of the problem could serve to clarify under which precise conditions this SDE can be used. The present paper aims to establish a first step in this direction.

The outline of this work is as follows. In section~\ref{qm} we summarize previous approaches that formally derive equation~\eqref{isde} using a
quantum mechanical formalism; we however do not follow exactly the path considered in the physical literature and construct our own viewpoint of this theory.
In section~\ref{gf} we re-derive these results using a more classical approach, that of generating functions.
In section~\ref{rnoise} we describe an explicitly solvable example of this theory that illustrates its correctness, at least in those cases
in which the noise is real. In section~\ref{notunreal} we approach the main issue of our work, i.e.~that of the validity of equation~\eqref{isde} as
an exact description of Markov chain~\eqref{master}, and show evidence that this indeed could be the case. Finally, we analyze the
stochastic dynamics of equation~\eqref{isde} in section~\ref{dynamics}, which developments complement those of  section~\ref{notunreal}.
Our main conclusions are drawn in section~\ref{conclusions}.

\section{Quantum Mechanics of Chemical Kinetics: The Formal Approach}
\label{qm}

Our objective is to study the abstract chemical reaction
$$
A+ A \stackrel{\lambda}{\longrightarrow} \emptyset,
$$
understood as a continuous in time Markov chain.
We define
\begin{displaymath}
P_n(t) \, dt := \mbox{probability of having } n \mbox{ particles in the time interval } [t,t+dt).
\end{displaymath}
Then clearly $P_n \geq 0 \,\, \forall \, n= 0, 1, 2, \ldots$, and $\sum_{n=0}^\infty P_n=1$.
This probability distribution obeys the following forward Kolmogorov equation
\begin{equation}
\begin{aligned}\label{mark11}
\frac{d}{dt} P_n(t) = \frac{\lambda}{2} [(n+2)(n+1)P_{n+2}(t)-n(n-1)P_n(t)].
\end{aligned}
\end{equation}
This differential equation can be regarded as an infinite system of ordinary differential equations.
We will now build an alternative way of approaching this problem, in the hope it will facilitate its analysis.

\subsection{The Abstract Vector Space Representation}

Our first step will be to build an abstract representation that embodies in its formulation the elements of the Markov chain.
Such a theory is known as the \emph{Doi-Peliti formalism} in the physical literature~\cite{doi1,doi2,peliti,peliti1}.
We start considering the set of linearly independent vectors $\mathcal{B}=\left\lbrace \, \left| n\right> : n \in \mathbb{N}\cup \left\lbrace 0\right\rbrace \right\rbrace$,
and the vector space they span over the real numbers, which we will henceforth denote as $\mathcal{R}$. One can  think of this set as the canonical Schauder basis that spans the space $\mathbb{R}^\infty$
defined to be the vector space of all sequences of real numbers.
These vectors are the ``states'' of the physical theory: $\left| j \right>$ describes
the state of the system that corresponds to the existence of exactly $j$ particles in the Markov chain. The goal of this construction is
to describe the time evolution of the Markov chain in a formalism akin to that of quantum mechanics. In order to achieve this, this
vector space should be endowed with a scalar product, an operation needed to tackle the problem of measurements in quantum mechanics.
However, this will be somehow problematic within the present formalism.

Before getting into this issue, we need to add some more structure to this vector space. We shall define
the following operators:
\begin{defi}[Annihilation and creation operators]
We define the action of these linear operators through their action on the elements of the basis.
\begin{itemize}
\item The annihilation operator acts $a\left|n\right>:=n\left|n-1\right>$ if $n \ge 1$ and $a\left|0\right>:=0$.
\item The creation operator acts $c\left|n\right>:=\left|n+1\right>$.
\end{itemize}
\end{defi}
The definition of the annihilation and creation operators  immediately implies the following result.
\begin{lem}
The commutator of the annihilation and creator operators is the identity operator
\begin{displaymath}
\left[a,c\right] \equiv ac-ca=1.
\end{displaymath}
\end{lem}

\begin{proof}
Compute
\begin{eqnarray}\nonumber
ac\left|n\right> &=& a\left|n+1\right> = (n+1) \left|n\right>,
\\ \nonumber
ca\left|n\right> &=& cn\left|n-1\right> = nc\left|n-1\right> = n \left|n\right>,
\end{eqnarray}
and subtract both equations.
\end{proof}

\begin{remark}[Combinatorics and operators]
We can think of the annihilation as a combinatoric operation. If we are given $n$ particles and about to
annihilate (or for the same purpose remove) one of them, we can do it in $n$ possible ways with the obvious result of ending up with $n-1$ particles. One can regard this as the combinatoric meaning of $a\left|n\right>=n\left|n-1\right>$. On the contrary, there is only one way to create (or add) a new particle.
This interpretation allows to build an intuitive picture for the non-vanishing commutator: given a set of $n$ particles, there are $n+1$ ways of creating
and then annihilating one particle; however, there are only $n$ ways in which we can do the same operations in reversed order.
\end{remark}

If, as in quantum mechanics, we want the creation and annihilation operators to be adjoint of each other, i.e. $c=a^{\dagger}$,
we need to endow $\mathcal{R}$ with a scalar product. We assume the elements of $\mathcal{B}$ to be pairwise orthogonal and proceed
using the language of quantum mechanics and denoting our scalar product as $\left< \cdot \right. \! \left| \cdot \right>$.

\begin{lem}\label{lemnorm}
The normalization condition $\left<0\right.\!\left|0\right>=1$ implies
\begin{displaymath}
\left<m\right.\!\left|n\right>=n! \, \delta_{nm},
\end{displaymath}
where the Kronecker delta
\begin{displaymath}
\delta_{nm}:=\left\lbrace \begin{array}{ccc}
1 & \mbox{if} & n=m \\
0 & \mbox{if} & n\neq m
\end{array}
\right..
\end{displaymath}
\end{lem}

\begin{proof}
By definition of an adjoint operator
\begin{displaymath}
\left<m\right|a\left|n\right>=\left<n\right|a^{\dagger}\left|m\right>.
\end{displaymath}
Then
\begin{displaymath}
n\left<m\right.\!\left|n-1\right>=\left<n\right.\!\left|m+1\right>,
\end{displaymath}
and by renaming the  variable $m$
\begin{displaymath}
n\left<m-1\right.\!\left|n-1\right>=\left<n\right.\!\left|m\right>,
\end{displaymath}
which is valid for $n\geq 1$ and $m\geq 1$. Imposing $\left<0\right.\!\left|0\right>=1$ and using orthogonality it is easy to see that
$\left<m\right.\!\left|n\right>=n! \, \delta_{nm}$.
\end{proof}

\begin{remark}
First, the Schauder basis $\mathcal{B}$ is obviously orthogonal (by construction of the scalar product) but not orthonormal.
Second, the scalar product is only partially defined on $\mathcal{R} \times \mathcal{R}$; note its domain is
$\mathbb{H} \times \mathbb{H} \subsetneq \mathcal{R} \times \mathcal{R}$, where
\begin{displaymath}
\mathbb{H} = \left\lbrace  \left| \Psi \right> = \sum_{n=0}^{\infty} \alpha_n \left|n\right> : \| \Psi \| := \left( \sum_{n=0}^{\infty}n! \, |\alpha_n|^2 \right)^{1/2}
< \infty  \right\rbrace.
\end{displaymath}
However  the annihilation operator is not an inner operation in this subset of $\mathcal{R}$; in other words
\begin{displaymath}
\begin{array}{ccc}
a:&\mathbb{H} \longrightarrow \mathbb{H}
\end{array}
\end{displaymath}
is not well defined, because
$$
\| a\left|\Psi \right> \|^2=\sum_{n=0}^\infty n! \, (n+1)^2 \, |\alpha_{n+1}|^2
=\sum_{n=1}^\infty n! \, n \, |\alpha_{n}|^2=\sum_{n=0}^\infty n! \, n \, |\alpha_{n}|^2
$$
does not need to be finite.
Finally, the subspace $\mathbb{H}$ does not contain all vectors of the form $\left|\Psi \right>= \sum_{n=0}^\infty P_n \left|n \right>$,
because it is assumed that the sequence of coefficients $P_n$  belongs only to $\ell^1$ (since they are probabilities).
\end{remark}

As the remark above shows, the set of states of the form $\left|\Psi \right>= \sum_{n=0}^\infty P_n \left|n \right>$ cannot be endowed with a Hilbert
space structure under the considered scalar product. Therefore one has to consider the pair
$(\{\left|\Psi \right>\},\left< \cdot \right. \! \left| \cdot \right>)$ only as a \emph{formal} Hilbert space, a subtlety not always remarked
in the physical literature~\cite{peliti}. There is however a way out of this pitfall,
which is exactly the one employed in quantum mechanics: the introduction of a Gelfand triple~\cite{reedsimon}. However, given that we are considering the Doi-Peliti
formalism only for the sake of contextualization, we will not explore this direction in the present work.

The objective of this formalism is to work with the state vectors
\begin{equation}\label{abst1}
\left|\Psi(t)\right> :=\sum_{n=0}^{\infty} P_n(t) \left|n\right>,
\end{equation}
where $P_n(t)$ is taken to be the solution of the forward Kolmogorov equation~(\ref{mark11}).

\begin{thm}\label{thschrodinger}
Equation~\eqref{mark11} is equivalent to
\begin{equation}\label{abst2}
\dfrac{d}{dt}\left|\Psi(t)\right>=\dfrac{\lambda}{2}[1-(a^{\dagger})^2]a^2\left|\Psi(t)\right>.
\end{equation}
\end{thm}

\begin{proof}
Using equation~\eqref{mark11} we compute
\begin{eqnarray}\nonumber
\dfrac{d}{dt}\left|\Psi(t)\right> &=& \sum_{n=0}^{\infty} \dfrac{d P_n(t)}{dt} \left|n\right> \\ \nonumber
%&=& \frac{\lambda}{2} \sum_{n=0}^{\infty} [(n+2)(n+1)P_{n+2}(t)-n(n-1)P_n(t)] \left|n\right> \\ \nonumber
&=& \frac{\lambda}{2} \sum_{n=0}^{\infty} (n+2)(n+1)P_{n+2}(t) \left|n\right>
- \frac{\lambda}{2} \sum_{n=0}^{\infty} n(n-1)P_n(t) \left|n\right> \\ \nonumber
&=& \frac{\lambda}{2} \sum_{n=0}^{\infty} P_{n+2}(t) a^2 \left|n+2\right>
- \frac{\lambda}{2} \sum_{n=0}^{\infty} P_n(t) (a^{\dagger})^2 a^2 \left|n\right> \\ \nonumber
&=& \frac{\lambda}{2} \sum_{n=2}^{\infty} P_{n}(t) a^2 \left|n\right>
- \frac{\lambda}{2} \sum_{n=2}^{\infty} P_n(t) (a^{\dagger})^2 a^2 \left|n\right> \\ \nonumber
&=& \frac{\lambda}{2} \sum_{n=0}^{\infty} P_{n}(t) [a^2  - (a^{\dagger})^2 a^2] \left|n\right>  %\\ \nonumber
%&=& \frac{\lambda}{2} [1 - (a^{\dagger})^2]a^2 \sum_{n=0}^{\infty} P_{n}(t) \left|n\right> \\ \nonumber
= \frac{\lambda}{2} [1 - (a^{\dagger})^2]a^2 \left|\Psi(t)\right>.
\end{eqnarray}
\end{proof}

\begin{remark}
If we define $\hat{H}:=\dfrac{\lambda}{2}[1-(a^{\dagger})^2]a^2$ we can rewrite~\eqref{abst2} as
\begin{displaymath}
\dfrac{d}{dt}\left|\Psi(t)\right>=\hat{H} \left|\Psi(t)\right>,
\end{displaymath}
which can be seen as a Schr{\"o}dinger-like formulation. Note that this formulation should be regarded as an evolution in the vector space
$\mathcal{R}$ (or in its proper subspace $\ell^1$~\cite{rogers}), not in a Hilbert space (see the Remark following Lemma~\ref{lemnorm} and the subsequent discussion).
\end{remark}

\subsection{Coherent States}\label{coh}

The next step in the Doi-Peliti formalism is the introduction of the \emph{coherent states}.

\begin{defi}[Coherent states]\label{cs}
For any $\phi \in \left[0,\infty\right)$, we define the coherent state
\begin{equation}\label{coher_state}
\left|\phi\right>:=e^{-\phi} \sum_{n=0}^{\infty}\dfrac{\phi^n}{n!}\left|n\right>.
\end{equation}
\end{defi}

\begin{remark}
Note there is an ambiguity in the notation as $\left|m\right>$, $m \in \mathbb{N}\cup\left\lbrace 0\right\rbrace$, could denote either a coherent state or an element of $\mathcal{B}$;
nonetheless we believe which one we are referring to should be clear from the context.
\end{remark}

\begin{lem}\label{eigenvector}
Coherent states are eigenvectors of the annihilation operator. In particular, the eigenvalue of $\left|\phi\right>$ is $\phi$.
\end{lem}

\begin{proof}
First, note that all coherent states are elements of the vector space $\mathcal{R}$. Then compute
\begin{equation*}
\begin{aligned}
a\left|\phi\right> &= e^{-\phi} \sum_{n=0}^{\infty}\dfrac{\phi^n}{n!}a\left|n\right>
= e^{-\phi} \sum_{n=1}^{\infty}\dfrac{\phi^n}{(n-1)!} \left|n-1\right> \\ \nonumber
&= \phi \, e^{-\phi} \sum_{n=1}^{\infty}\dfrac{\phi^{n-1}}{(n-1)!} \left|n-1\right>
= \phi \, e^{-\phi} \sum_{n=0}^{\infty}\dfrac{\phi^{n}}{n!} \left|n\right>
&= \phi \left|\phi\right>.
\end{aligned}
\end{equation*}
\end{proof}

The coherent state $\left| \phi \right>$, defined in \eqref{coher_state}, % $=\sum_{n=0}^{\infty} \frac{\phi^n}{n!}e^{-\phi}\left| n \right>$
represents a Poisson distribution
with parameter $\phi$. At least formally, we can write the integral representation
\begin{equation}\label{prepresentation}
\left| \Psi \right>= \int_{0}^{\infty} \Psi(\phi) \left| \phi \right> d\phi,
\end{equation}
for some amplitude $\Psi(\phi)$,
where the integral has to be regarded as a Pettis integral, i.e.~the duality product induced by the scalar product
does commute with the integral~\cite{talagrand}. This representation could be seen as an expansion of a state vector in terms of all different Poissonians (and hence the name Poisson representation~\cite{gardiner,gardiner0,gardiner1}).
Now consider the vector $\left| \Psi \right>=\sum_{n=0}^\infty P_n \left| n \right>$ to find
\begin{equation}\label{prepresentation2}
P_n=\dfrac{1}{n!}\int_0^{\infty} \phi^n \, e^{-\phi} \, \Psi(\phi) \, d\phi,
\end{equation}
which should be regarded as the coordinatewise meaning of \eqref{prepresentation}.
It is not clear that any vector in $\mathcal{R}$ can be represented in this form.
Our next step will be to show that actually we need to allow $\Psi(\phi)$ to take values in a space of distributions.
Let us remind the reader the definition of $\ell^1$:
\begin{displaymath}
\ell^1 := \left\lbrace \{\alpha_n\}_{n=0}^\infty \in \mathcal{R} \, \left\vert \,\, \sum_{n=0}^\infty  |\alpha_n| <\infty \right. \right\rbrace.
\end{displaymath}
Now define the operator
\begin{eqnarray}\nonumber
T: L^1_+(0,\infty) &\longrightarrow& \ell^1 \\ \nonumber
\Psi(\phi) &\longmapsto&
T[\Psi(\phi)]:= \left\{ \dfrac{1}{n!}\int_0^{\infty} \phi^n \, e^{-\phi} \, \Psi(\phi) \, d\phi \right\}_{n=0}^\infty,
\end{eqnarray}
where $L^1_+(0,\infty)= \{ \Theta \in L^1(0,\infty) : \Theta(\cdot) \ge 0 \}$.
The following result accounts for the properties of $T$.

\begin{prop}\label{propt}
The operator $T$ is well-defined, linear, continuous, and not surjective.
\end{prop}

\begin{proof}
It is clear that this operator is linear if well-defined. To see it is well-defined and continuous compute
\begin{eqnarray}\nonumber
\Vert T(\Psi) \Vert_{\ell^1} &=& \sum_{n=0}^{\infty} \dfrac{1}{n!}\left\vert\int_0^{\infty} \phi^n \, e^{-\phi} \, \Psi(\phi) \, d\phi \right\vert = \sum_{n=0}^{\infty} \dfrac{1}{n!}\int_0^{\infty} \phi^n \, e^{-\phi} \left\vert \Psi(\phi)\right\vert d\phi \\ \nonumber
&=& \int_0^{\infty} \sum_{n=0}^{\infty} \dfrac{1}{n!} \, \phi^n \, e^{-\phi} \left\vert \Psi(\phi)\right\vert d\phi = \int_0^{\infty} \left\vert \Psi(\phi)\right\vert d\phi = \Vert \Psi \Vert_{L^1} <\infty,
\end{eqnarray}
where we have used the monotone convergence theorem in order to commute the integral and the sum.

To show that $T$ is not surjective, we will show that there is no $L^1_+$ function which image is
$e_0=(1,0,0,0....)$ ($\equiv \left|0\right>$).
Lets proceed by contradiction: suppose that there exists such a function $\Psi_0 \in L^1_+(0,\infty)$ with $T(\Psi_0)=e_0$, then
\begin{displaymath}
\dfrac{1}{n!} \int_0^{\infty} \phi^n \, e^{-\phi} \, \Psi_0(\phi) \, d\phi = \delta_{n 0},
\end{displaymath}
or equivalently
\begin{displaymath}
\int_0^{\infty} \phi^n \, e^{-\phi} \, \Psi_0(\phi) \, d\phi = \delta_{n 0}.
\end{displaymath}
Therefore for any polynomial $P(\phi)$ we have
\begin{displaymath}
\int_0^{\infty} P(\phi) \, e^{-\phi} \, \Psi_0(\phi) \, d\phi = P(0),
\end{displaymath}
which  implies
\begin{displaymath}
\lvert P(0) \rvert \leq \int_0^{\infty} \lvert P(\phi) \rvert \, e^{-\phi} \, \Psi_0(\phi) \, d\phi.
\end{displaymath}
Writing $P(\phi)=\sum_{n=0}^N \alpha_n \phi^n$ for  $\alpha_n \in \mathbb{R}$, yields
\begin{displaymath}
\lvert P(\phi) \rvert \leq \sum_{n=0}^N \lvert \alpha_n \rvert \, \phi^n =: Q(\phi) \,\, \quad  \forall \, \phi \geq 0.
\end{displaymath}
As $Q(\phi)$ is a polynomial and $Q(0) = \lvert \alpha_0 \rvert = \lvert P(0) \rvert $ then, following our assumption,
\begin{displaymath}
\int_0^{\infty} \lvert P(\phi) \rvert \, e^{-\phi} \, \Psi_0(\phi) \, d\phi  \leq \int_0^{\infty} Q(\phi) \, e^{-\phi} \, \Psi_0(\phi) \, d\phi = Q(0) = \lvert P(0) \rvert,
\end{displaymath}
which in turn implies
\begin{equation}\label{P_pol}
\int_0^{\infty} \lvert P(\phi) \rvert \, e^{-\phi} \, \Psi_0(\phi) \, d\phi  = \lvert P(0) \rvert
\end{equation}
for all $P$ polynomials.

Now assume $f(\phi) \in C_c(\mathbb{R}_+)$ and compute
\begin{eqnarray}\nonumber
\int_0^{\infty} f(\phi) \, e^{-\phi} \, \Psi_0(\phi) \, d\phi &=& \int_0^{L} f(\phi) \, e^{-\phi} \, \Psi_0(\phi) \, d\phi \\ \nonumber
&=& \int_0^{L} P(\phi) \, e^{-\phi} \, \Psi_0(\phi) \, d\phi
+ \int_0^{L} \left[ f(\phi)-P(\phi) \right] e^{-\phi} \, \Psi_0(\phi) \, d\phi \\ \nonumber
&\le& \int_0^{L} |P(\phi)| \, e^{-\phi} \, |\Psi_0(\phi)| \, d\phi
+ \int_0^{L} |f(\phi)-P(\phi)| \, e^{-\phi} \, |\Psi_0(\phi)| \, d\phi
\\ \nonumber
&\le& \int_0^{\infty} |P(\phi)| \, e^{-\phi} \, |\Psi_0(\phi)| \, d\phi
+ \int_0^{L} |f(\phi)-P(\phi)| \, e^{-\phi} \, |\Psi_0(\phi)| \, d\phi
\\ \nonumber
&\le& |P(0)| + \|f(\phi)-P(\phi)\|_{L^\infty(0,L)} \|\Psi_0(\phi)\|_{L^1(0,\infty)} \\ \nonumber
&=& |P(0)-f(0)| + \|f(\phi)-P(\phi)\|_{L^\infty(0,L)} \|\Psi_0(\phi)\|_{L^1(0,\infty)},
\end{eqnarray}
where $L>0$ is large enough so $[0,L]$ contains the support of $f(\phi)$ and we have used that $f(0)=0$.
The Weierstrass approximation theorem, together with \eqref{P_pol},  assures us that we can choose a polynomial $P(\phi)$ such that
$$
\int_0^{\infty} f(\phi) \, e^{-\phi} \, \Psi_0(\phi) \, d\phi \le \epsilon \quad \forall \, \epsilon >0.
$$
An analogous argument yields the reversed inequality, so we conclude
$$
\int_0^{\infty} f(\phi) \, e^{-\phi} \, \Psi_0(\phi) \, d\phi = 0,
$$
and this equality holds for any $f(\phi) \in C_c(\mathbb{R}_+)$. This implies that $\Psi_0(\phi)=0$ a.e., and hence a contradiction.
\end{proof}

\begin{remark}
This proposition, in particular, states that the set of $L^1_+(0, \infty)$ functions is not sufficient to describe the space $\ell^1$ completely
via the representation~\eqref{prepresentation}.
\end{remark}

Although we do not have a clear characterization of the Poisson representation we move forward
to introduce time dependence in it:
\begin{displaymath}
\left| \Psi(t) \right> = \int_{0}^{\infty} \Psi(\phi,t) \left| \phi \right> d\phi.
\end{displaymath}
In order to partially characterize the time evolution of this amplitude we need the following technical result.

\begin{lem}
The coherent states fulfill the following  properties
\begin{eqnarray}\nonumber
a\left|\phi\right> &=& \phi \left|\phi\right>, \\ \nonumber
a^{\dagger} \left|\phi\right> &=& \left(1+\dfrac{\partial}{\partial \phi}\right)\left|\phi\right>.
\end{eqnarray}
\end{lem}

\begin{proof}
The first property was already proven in Lemma~\ref{eigenvector}. To find the second compute
\begin{eqnarray}\nonumber
a^{\dagger}\left|\phi\right> &=& e^{-\phi} \sum_{n=0}^{\infty}\dfrac{\phi^n}{n!}a^{\dagger}\left|n\right> = e^{-\phi} \sum_{n=0}^{\infty}\dfrac{\phi^n}{n!} \left|n+1\right> \\ \nonumber
&=& e^{-\phi} \dfrac{\partial}{\partial \phi} \sum_{n=0}^{\infty}\dfrac{\phi^{n+1}}{(n+1)!}\left|n+1\right> = e^{-\phi} \dfrac{\partial}{\partial \phi} \sum_{n=1}^{\infty}\dfrac{\phi^{n}}{n!}\left|n\right> \\ \nonumber
&=& e^{-\phi} \dfrac{\partial}{\partial \phi} \sum_{n=0}^{\infty}\dfrac{\phi^{n}}{n!}\left|n\right> = \left(1+\dfrac{\partial}{\partial \phi}\right) \left( e^{-\phi} \sum_{n=0}^{\infty}\dfrac{\phi^{n}}{n!}\left|n\right> \right) = \left(1+\dfrac{\partial}{\partial \phi}\right)\left|\phi\right>.
\end{eqnarray}
\end{proof}

Now we are ready to partially characterize the time evolution of the amplitude $\Psi(\phi,t)$ by means of a partial differential equation.

\begin{thm}\label{bfpe}
If $\Psi(\phi,t)$ is a bounded $C^1(0,\infty;C^2(0,\infty))$ solution to the partial differential equation
\begin{equation}\label{bad1}
\dfrac{\partial \Psi}{\partial t}=\dfrac{\lambda}{2}\left(2\dfrac{\partial}{\partial \phi}-\dfrac{\partial^2}{\partial \phi^2}\right)
\left( \phi^2 \Psi \right) \;\; \text{ for } \,\, t>0 \; \text{ and } \;  \phi\in (0, \infty),
\end{equation}
then $P_n(t)$ defined by \eqref{prepresentation2} satisfies~\eqref{mark11}.
\end{thm}

\begin{proof}
From equation~\eqref{prepresentation2} we have
\begin{equation}\nonumber
\dfrac{d}{dt}P_n(t) = \frac{1}{n!} \int_{0}^{\infty} \dfrac{\partial \Psi(\phi,t)}{\partial t} e^{-\phi} \phi^n d\phi, \\ \nonumber
\end{equation}
or alternatively in vector form
\begin{equation}\nonumber
\dfrac{d}{dt}\left|\Psi(t)\right> = \int_{0}^{\infty} \dfrac{\partial \Psi(\phi,t)}{\partial t} \left| \phi \right> d\phi. \\ \nonumber
\end{equation}
By Theorem~\ref{thschrodinger} equation~\eqref{mark11} is equivalent to~\eqref{abst2}, and so
\begin{eqnarray}\nonumber
\dfrac{d}{dt}\left|\Psi(t)\right> &=& \dfrac{\lambda}{2}[1-(a^{\dagger})^2]a^2\left|\Psi(t)\right> = \dfrac{\lambda}{2} \int_{0}^{\infty} \Psi(\phi,t) [1-(a^{\dagger})^2]a^2 \left| \phi \right> d\phi \\ \nonumber
&=& \dfrac{\lambda}{2} \int_{0}^{\infty} \Psi(\phi,t) [1-(a^{\dagger})^2]\phi^2 \left| \phi \right> d\phi
= \dfrac{\lambda}{2} \int_{0}^{\infty} \phi^2 \, \Psi(\phi,t) [1-(a^{\dagger})^2] \left| \phi \right> d\phi \\ \nonumber
&=& -\dfrac{\lambda}{2} \int_{0}^{\infty} \phi^2 \, \Psi(\phi,t) \left( 2\dfrac{\partial}{\partial \phi} + \dfrac{\partial^2}{\partial \phi^2} \right)
\left| \phi \right> d\phi
= \dfrac{\lambda}{2} \int_{0}^{\infty}
\left( 2\dfrac{\partial}{\partial \phi} - \dfrac{\partial^2}{\partial \phi^2} \right) \left[ \phi^2 \, \Psi(\phi,t) \right]
\left| \phi \right> d\phi,
\end{eqnarray}
where the last step has to be understood componentwise, using  the integration by parts,
\begin{equation}\nonumber
-\dfrac{\lambda}{2} \int_{0}^{\infty} \phi^2 \, \Psi(\phi,t) \left( 2\dfrac{\partial}{\partial \phi} + \dfrac{\partial^2}{\partial \phi^2} \right)
\left( \phi^n \, e^{-\phi} \right) d\phi
= \dfrac{\lambda}{2} \int_{0}^{\infty}
\left( 2\dfrac{\partial}{\partial \phi} - \dfrac{\partial^2}{\partial \phi^2} \right) \left[ \phi^2 \, \Psi(\phi,t) \right]
\phi^n \, e^{-\phi} d\phi.
\end{equation}
Therefore
\begin{displaymath}
\int_{0}^{\infty}\left[ \dfrac{\partial \Psi}{\partial t}-\dfrac{\lambda}{2}\left(2\dfrac{\partial}{\partial \phi}-\dfrac{\partial^2}{\partial \phi^2}\right)
\left( \phi^2 \Psi\right) \right] \left|\phi\right> d\phi=0,
\end{displaymath}
i.e.
\begin{displaymath}
\int_{0}^{\infty}\left[ \dfrac{\partial \Psi}{\partial t}-\dfrac{\lambda}{2}\left(2\dfrac{\partial}{\partial \phi}-\dfrac{\partial^2}{\partial \phi^2}\right)
\left( \phi^2 \Psi\right) \right] \phi^n \, e^{-\phi} d\phi=0.
\end{displaymath}
\end{proof}

We can rewrite equation~\eqref{bad1} in the following way
\begin{equation}\label{bad2}
\dfrac{\partial \Psi}{\partial t}=\lambda\left(\dfrac{\partial}{\partial \phi}+\dfrac{i^2}{2}\dfrac{\partial^2}{\partial \phi^2}\right)
\left( \phi^2 \Psi \right),
\end{equation}
where $i$ is the imaginary unit. If we formally regarded this equation as a Fokker-Planck equation, we could be tempted to study
the formally associated SDE
\begin{equation}\label{complex1}
d \phi = - \phi^2 dt + i \, \phi \, dW_t,
\end{equation}
obtained after rescaling time $t \to t/\lambda$. We call this apparently magic step, which we have found nowhere justified within
the framework of probability theory, the \emph{imaginary It\^o interpretation} of equation~\eqref{bad1}.
Before continuing it is important to highlight the following facts:
\begin{itemize}
\item Equation~\eqref{bad1} (in its original formulation or written in the form~\eqref{bad2}) is \emph{not} a Fokker-Planck equation, since it has negative diffusion~\cite{risken}.
\item Equation~\eqref{bad1} is a backward diffusion equation and therefore ill-posed if considered forward in time, at least if the problem
is posed in usual functional spaces~\cite{renar}.
\item Nevertheless~\eqref{bad1} could possibly be considered as a well-posed equation in a suitable distributional space~\cite{aragaosilva}.
This of course would make $\Psi$ distribution-valued and therefore not interpretable as a probability measure.
We illustrate this fact by means of an explicit solution in Appendix~\ref{expformula}.
\end{itemize}
Although it looks like the imaginary It\^o interpretation is a purely formal and possibly ill-defined step,
our present objective is to show that this is \emph{not always} the case and that equation~\eqref{complex1} \emph{could} represent the
Markov process described by~\eqref{mark11} in a certain sense.

\section{Generating Functions: A Classical Approach}
\label{gf}

In this section we consider an alternative approach based on generating functions~\cite{wilf}.
This theory is closely related to the previous one, but perhaps one could say that it uses a more standard mathematical machinery.
First of all we note the following equivalences:
\begin{displaymath}
\left\lbrace
\begin{array}{ccc}
\left|n\right> & \longleftrightarrow & x^n, \\
a & \longleftrightarrow & \dfrac{d}{dx}\left( \cdot \right), \\
a^{\dagger} & \longleftrightarrow & x \times \left( \cdot \right).
\end{array}
\right.
\end{displaymath}
It is clear that this formalism is coherent with the one introduced previously due to the properties:
\begin{displaymath}
\left\lbrace
\begin{array}{l}
\dfrac{d}{dx}x^n=nx^{n-1}, \\
x\cdot x^n=x^{n+1}.
\end{array}
\right.
\end{displaymath}
Also the value for the commutator follows directly
\begin{displaymath}
\left[ \dfrac{d}{dx},x\right]=1.
\end{displaymath}
The ``state'' of our physical system will now be encoded in an analytic function
\begin{displaymath}
G(x):=\sum_{n=0}^\infty P_n \, x^n,
\end{displaymath}
where $P_n$ are the probabilities and hence $P_n \geq 0 $  for all $n \in \mathbb N \cup \{0\}$ and $\sum_{n=0}^\infty P_n=1$.
Clearly, $G \in C^{\omega}\left( -1,1\right)\cap C\left[ -1,1\right]$, i.e. $G$ is analytic in the open interval and continuous in its closure.
It is also possible to consider $G$ as a function of a complex variable $z \in \mathbb{S}^1$ and in that case
$G$ would be a holomorphic function in the complex open unit disk with continuous closure in $\mathbb{S}^1$.

We can move to the time-dependent formalism via the introduction of the time-dependent generating function
\begin{equation}\label{maineq}
G(t,x):=\sum_{n=0}^\infty P_n(t) \, x^n  \quad  \text{ for } \; \;  t\geq 0 \;   \text{ and } \;  x\in \left[ -1,1\right],
\end{equation}
which is an analog of~(\ref{abst1}).

\begin{thm}\label{thm_3}
The time-dependent generating function  \eqref{maineq} satisfies the partial differential equation
\begin{equation}\label{gfpde}
\dfrac{\partial G}{\partial t}=\dfrac{\lambda}{2}\left(1-x^2\right)\dfrac{\partial^2 G}{\partial x^2} \qquad \text{ for } \; t >0 \; \text{ and } \; x\in (-1,1)
\end{equation}
if and only if $P_n(t)$ satisfies the forward Kolmogorov equation~\eqref{mark11}.
\end{thm}

\begin{proof}
By means of equation~\eqref{mark11} and  for $x\in (-1,1)$ we have
\begin{eqnarray}\nonumber
\dfrac{\partial}{\partial t} G(x,t) &=& \sum_{n=0}^{\infty} \dfrac{d P_n(t)}{dt} \, x^n \\ \nonumber
%&=& \frac{\lambda}{2} \sum_{n=0}^{\infty} [(n+2)(n+1)P_{n+2}(t)-n(n-1)P_n(t)] \, x^n \\ \nonumber
&=& \frac{\lambda}{2} \sum_{n=0}^{\infty} (n+2)(n+1)P_{n+2}(t) \, x^n
- \frac{\lambda}{2} \sum_{n=0}^{\infty} n(n-1)P_n(t) \, x^n \\ \nonumber
&=& \frac{\lambda}{2} \sum_{n=0}^{\infty} P_{n+2}(t) \, \frac{d^2}{dx^2} \, x^{n+2}
- \frac{\lambda}{2} \sum_{n=0}^{\infty} P_n(t) \, x^2 \frac{d^2}{dx^2} \, x^n \\ \nonumber
&=& \frac{\lambda}{2} \sum_{n=2}^{\infty} P_{n}(t) \, \frac{d^2}{dx^2} \, x^n
- \frac{\lambda}{2} \sum_{n=2}^{\infty} P_n(t) \, x^2 \frac{d^2}{dx^2} \, x^n \\ \nonumber
&=& \frac{\lambda}{2} \sum_{n=0}^{\infty} P_{n}(t) \left(\frac{d^2}{dx^2}  - x^2 \frac{d^2}{dx^2}\right) x^n \\ \nonumber
&=& \frac{\lambda}{2} \left(1 - x^2\right) \frac{\partial^2}{\partial x^2} \sum_{n=0}^{\infty} P_{n}(t) x^n = \frac{\lambda}{2} \left(1 - x^2 \right) \frac{\partial^2}{\partial x^2} G(x,t).
\end{eqnarray}
Notice that  for $x\in (-1,1)$ the series in \eqref{maineq} and the corresponding series for the derivatives with respect to $x$ and $t$ are absolutely convergent and, hence,  we can interchange the order of differentiation and summation.
\end{proof}

\begin{remark}
The result in Theorem~\ref{thm_3} is the analog of the result in Theorem~\ref{thschrodinger}.
This result is extended to any arbitrary reaction in Appendix~\ref{sectgeneral}.
\end{remark}

It is important to note that from the initial conditions for $P_n$, i.e. $P_n(0)$,
we obtain only the initial value $G(0,x) = \sum_{n=0}^\infty P_n(0) \, x^n$
for equation~\eqref{gfpde}, but not the boundary conditions.
The lack of boundary conditions comes from the degeneration of the elliptic operator in~\eqref{gfpde} at the boundary,
which prevents the evolution of the boundary values.
This fact has a probabilistic meaning too, as it encodes the existence of two conserved quantities:
\begin{itemize}
\item \emph{Conservation of probability}: $G(t,1)=1$ for all $t\geq 0$.
\item \emph{Conservation of parity}: $G(t,-1)=\wp$ for all $t\geq 0$, where $\wp = \sum_{n=0}^\infty P_{2n}(0) - \sum_{n=0}^\infty P_{2n+1}(0)$.
\end{itemize}
While the existence of the first conserved quantity is ensured  for every type of reactions, the existence of the second one is a particular
consequence of the structure of the binary annihilation $A + A \to \emptyset$. Its intuitive meaning becomes clear when we consider
an initial condition of the type $\delta_{nm}$, i.e. the initial number of particles is fixed for every realization of the stochastic process.
Then $\wp=1$ if $m$ is even and $\wp=-1$ if $m$ is odd. In the same way, if $P_n(0) >0$ for some even and odd values of $n$, then the probability
of finding an even or odd number of particles at an arbitrary time is the same as initially (provided we assume that $0$ is an even number).

One of the advantages of the formalism of generating functions is that it allows us to recover, in a direct way, the probabilities
\begin{displaymath}
P_n(t)=\dfrac{1}{n!}\dfrac{\partial^n G}{\partial x^n}(t,0),
\end{displaymath}
as well as the cumulants
\begin{displaymath}
\mathbb{E}\left[n(n-1)\cdots(n-m+1)\right](t)=\dfrac{\partial^m G}{\partial x^m}(t,1).
\end{displaymath}
However, in order to connect the theory related to the generating functions to the one described in the previous section, we need to %consider the coherent states and
define the corresponding coherent generating function.

\begin{defi}[Coherent generating function]\label{def_3}
For any parameter $\phi \in [0,\infty)$, we define the coherent generating function $G: \mathbb{R} \longrightarrow (0, \infty)$ as
\begin{displaymath}
G_{\phi}(x):=e^{\phi (x-1)}.
\end{displaymath}
\end{defi}

\begin{remark}
Notice  the relation between  the definition of  the coherent generating function, Definition~\ref{def_3},  and the definition of coherent states in  Definition~\ref{cs}.
\end{remark}

Then the analog to the representation~\eqref{prepresentation} in the  context of generating functions  is given by
\begin{equation}\label{G_2}
G(t,x)=\int_{0}^{\infty}\Psi(t,\phi) \, e^{\phi (x-1)} \, d\phi \qquad \text{ for } \; t \geq 0 \; \text{and } \;  x\in [-1,1] \text{, with } \Psi(t,\cdot) \in L^1(0,\infty).
\end{equation}
Using  equation~\eqref{gfpde} for $G$ we can  determine the equation  for the amplitude $\Psi$, and hence further illustrate
the form for $G$ considered in~\eqref{G_2}.

\begin{thm}\label{thm_4}
If $\Psi \in C^1(\left[0, \infty\right); C^2(0, \infty) \cap L^1(0,\infty))$, with $\Psi(t,\phi)$ bounded in $\phi$ for all $t$,  is a solution of   equation \eqref{bad1},
% \begin{equation}\label{eq_Psi}
% \dfrac{\partial \Psi}{\partial t}=\dfrac{\lambda}{2}\left(2\dfrac{\partial}{\partial \phi}-\dfrac{\partial^2}{\partial \phi^2}\right)
% \left( \phi^2 \Psi \right)  \qquad \text{ for } \; t > 0 \; \text{ and } \;  \phi \in (0, \infty),
% \end{equation}
then  $G\in C^1(0, \infty; C^{\infty}(-1,1))\cap C([0, \infty); C[-1,1])$, given by \eqref{G_2}, satisfies equation~\eqref{gfpde}.
\end{thm}

\begin{proof} Notice that for $x \in (-1,1)$ the integral in \eqref{G_2}, together with  its derivatives with respect to
$x$ (of first and second order) and $t$ (of first order), are well-defined.
Then formula~\eqref{G_2} implies
\begin{equation}\nonumber
\dfrac{\partial}{\partial t} G(t,x) = \int_{0}^{\infty} \dfrac{\partial \Psi(t,\phi)}{\partial t} \, e^{\phi (x-1)} \, d\phi.
\end{equation}
Since  $\Psi$ is a solution of \eqref{bad1},    we obtain  for $x\in (-1, 1)$ that
\begin{equation}\nonumber
\begin{aligned}
\dfrac{\partial}{\partial t}G(t,x) &=\;  \dfrac{\lambda}{2} \int_{0}^{\infty}
\left( 2\dfrac{\partial}{\partial \phi} - \dfrac{\partial^2}{\partial \phi^2} \right) \left[ \phi^2 \, \Psi(t,\phi) \right]
e^{\phi (x-1)} \, d\phi \\
&= -\dfrac{\lambda}{2} \int_{0}^{\infty} \phi^2 \, \Psi(t,\phi) \left( 2\dfrac{\partial}{\partial \phi} + \dfrac{\partial^2}{\partial \phi^2} \right)
e^{\phi (x-1)} \, d\phi \\
& %= \dfrac{\lambda}{2} \int_{0}^{\infty} \phi^2 \, \Psi(\phi,t) \left(1-x^2\right) e^{\phi (x-1)} \, d\phi
=\;  \dfrac{\lambda}{2} \int_{0}^{\infty} \Psi(t,\phi) \left(1-x^2\right) \phi^2 \, e^{\phi (x-1)} \, d\phi \\
& = \; \dfrac{\lambda}{2} \int_{0}^{\infty} \Psi(t,\phi) \left(1-x^2\right) \frac{\partial^2}{\partial x^2} \, e^{\phi (x-1)} \, d\phi \\
& = \;  \dfrac{\lambda}{2}\left(1-x^2\right) \frac{\partial^2}{\partial x^2} G(t,x).
\end{aligned}
\end{equation}
Notice that for $x\in (-1, 1)$ and $\phi \in [0, \infty)$ we have that the  terms  $\phi^2 \Psi(t,\phi) e^{\phi (x-1)} $, $\phi^2 \partial_\phi \Psi(t,\phi)
e^{\phi (x-1)}$,  $\phi \Psi(t,\phi) e^{\phi (x-1)}$, and $\phi^2 \Psi(t,\phi)(x-1)^2 e^{\phi (x-1)}$  converge to $0$ as $\phi \to +\infty$, for all $t >0$. Hence all boundary terms obtained due to integration by parts vanish.
Regularity for $G$ follows from \eqref{G_2} and the regularity for $\Psi$.
\end{proof}

%Hence in~\eqref{eq_Psi} we recover again equation~(\ref{bad1}), which is formally related to~\eqref{complex1}.
Notice that despite the singular character
of~\eqref{bad1}, which prevents the construction of a classical existence and uniqueness theory, such a theory can be built for equation~\eqref{gfpde}.
We start by defining the weighted Sobolev space $H_\rho^1(-1,1)$ as
$$
H_\rho^1(-1,1) = \left\{ v \in L^2(-1,1) \; : \; \sqrt{1-x^2} \,  \frac{\partial v}{ \partial x}  \in L^2(-1,1)\right \}.
$$
We first state the uniqueness result.

\begin{thm} Consider   $G_0 \in L^2(-1,1)$ and $\wp \in \mathbb R$.
	There exists at most one solution \\ $G\in C(0, \infty; L^2(-1,1))\cap L^2_{\rm loc}(0, \infty; H_\rho^1(-1,1))$, with  $G(\cdot,1), G(\cdot,-1) \in L^2_{\rm loc} (0, \infty)$, of the problem
	\begin{equation}\label{eq_G_unique}
	\begin{cases}
	\dfrac{\partial G}{\partial t} =\dfrac{\lambda}{2}\left(1-x^2\right)\dfrac{\partial^2 G}{\partial x^2} & \text{ for } \;  t>0 \;  \text{ and } \;  x\in (-1,1),  \\
	G(0,x) =G_0(x) & \text{ for } \; x\in (-1,1), \\
	G(t,1) =1, \quad  G(t,-1) = \wp \quad & \text{ for } \; t >0.
	\end{cases}
	\end{equation}
\end{thm}

\begin{proof}
	Assume that  there are two solutions  $G_1$ and $G_2$ of \eqref{eq_G_unique}.  Then  $\overline G = G_1 - G_2$ satisfies
	\begin{equation}\nonumber
	\begin{aligned}
	& \int_{-1}^1|\overline G(\tau,x)|^2 dx +
	\lambda \int_0^\tau \int_{-1}^1 \left( (1-x^2) \big|\partial_x \overline G(t,x)\big|^2 +   |\overline G(t,x)|^2    \right)  dx dt \\
	= & \int_{-1}^1|\overline G(0,x)|^2 dx +  \lambda  \int_0^\tau \left(|\overline G(t,1)|^2  + |\overline G(t,-1)|^2\right) dt
	\end{aligned}
	\end{equation}
	for $\tau \in (0, \infty)$.  Here we used that
	$$
	\begin{aligned}
	(1-x^2)\frac{\partial^2 G}{\partial^2 x} & = \frac{\partial }{\partial x} \left((1- x^2)    \frac{\partial G}{\partial x} \right) + 2 x  \frac{\partial G}{\partial x},
	\\
	2 x  \frac{\partial G}{\partial x} G&  = \frac{\partial }{\partial x} \left( x  |G|^2 \right) - |G|^2.
	\end{aligned}
	$$
	Then $\overline G(0,x) =0$ for $x\in (-1,1)$, and   $\overline G(t,1) =0$,  $\overline G(t,-1) =0$ for   $t>0$ yield
	$$
	\begin{aligned}
	\int_{-1}^1 |\overline G(\tau,x)|^2 dx + \lambda  \int_0^\tau \int_{-1}^1 \left[ (1-x^2) \left|\partial_x \overline G(t,x)\right|^2  +  |\overline G(t,x)|^2 \right] dx dt =0
	\end{aligned}
	$$
	for any $\tau >0$.
	Thus we obtain that  $\overline G(t, x) =0$, and hence $G_1(t,x) = G_2(t,x)$,  for $t\geq 0$ and $x\in [-1,1]$.
\end{proof}

Now we move to the problem of existence.
\begin{thm} Assume that  $G_0 \in L^2(-1,1)$ and $\wp \in \mathbb R$.
	Then there exists a  solution $G\in C(0, \infty; L^2(-1,1))\cap L^2_{\rm loc} (0, \infty; H^1_\rho(-1,1))$, with $\partial_t G \in L^2_{\rm loc} (0,\infty; H^{-1}(-1,1))$,  of  problem \eqref{eq_G_unique}.
	If $G_0 \in H^1(-1,1)$ then  $\partial_x G \in L^2_{\rm loc} (0, \infty; H^1_\rho(-1,1))$,  $G \in C(0, \infty; H^1(-1,1))$, and $\partial_t G \in L^2_{\rm loc} (0, \infty; L^2(-1,1))$.
\end{thm}
\begin{proof}
	Applying the Galerkin method, together with a priori estimates derived below, ensures the existence of a solution $G\in C(0, \infty; L^2(-1,1))\cap L^2(0, \infty; H^1_\rho(-1,1))$ of  problem \eqref{eq_G_unique}.
	
	Considering $G$ as a test function for equation in \eqref{eq_G_unique} we obtain
	\begin{equation}\nonumber
	\begin{aligned}
	& \int_{-1}^1|G(\tau,x)|^2 dx +
	\lambda \int_0^\tau \int_{-1}^1 \left( (1-x^2) \big|\partial_x  G(t,x)\big|^2 +   |G(t,x)|^2    \right)  dx dt \\
	= & \int_{-1}^1|G(0,x)|^2 dx +  \lambda  \int_0^\tau \left(| G(t,1)|^2  + |G(t,-1)|^2\right) dt, \,\, \quad  \text{ for } \tau \in (0, \infty).
	\end{aligned}
	\end{equation}
	Thus the assumptions on initial and boundary conditions ensure
	$$
	\sup\limits_{t \in (0,T)}   \int_{-1}^1|G(t,x)|^2 dx + \int_0^T \int_{-1}^1  (1-x^2) \big|\partial_x  G(t,x)\big|^2 dx dt \leq C
	$$
	for any $T>0$ and a  constant $C>0$.  Then from equation \eqref{eq_G_unique} we also obtain that
	$$
	\|\partial_t G \|_{L^2(0,T; H^{-1}(-1,1))} \leq C.
	$$
	Differentiating the equation in \eqref{eq_G_unique}  with respect to $x$ and taking $\partial_x G$ as a test function we obtain
	$$
	\begin{aligned}
	\int_{-1}^1 |\partial_x G(\tau,x)|^2 dx + \lambda \int_0^\tau \int_{-1}^1 (1-x^2) \big|\partial^2_x  G(t,x)\big|^2 dx dt = \int_{-1}^1 |\partial_x G(0,x)|^2 dx
	\end{aligned}
	$$
	for $\tau>0$.
	
	Then if $G_0 \in H^1(-1,1)$ we obtain $\partial_x G \in L^\infty(0, T; L^2(-1,1))$ and $\sqrt{1-x^2} \, \partial^2_x  G \in L^2(0, T; L^2(-1,1))$ for any $T \in (0, \infty)$. From  equation   in \eqref{eq_G_unique}  we obtain also that   $\partial_t G  \in L^2(0,T; L^2(-1,1))$ and  $\partial_t \partial_x G \in L^2(0, T; (H^1(-1,1))^\prime)$ for all $T >0$.
	Hence $G \in C(0, \infty; H^1(-1,1))$.	
\end{proof}

\section{Imagine the Noise were Real}
\label{rnoise}

In order to perform a step forward towards the understanding of the coherent representation we will analyze simpler reaction schemes that do not produce
an imaginary noise within the framework of coherent state PDEs. We start with the simpler case in which no noise is present and subsequently move to the case of real noise.

\subsection{No Noise}

Consider the abstract reaction
\begin{displaymath}
A\stackrel{\lambda}{\longrightarrow} \varnothing.
\end{displaymath}
The corresponding forward Kolmogorov equation reads
\begin{displaymath}
\dfrac{d P_n}{dt}=\lambda \{ (n+1)P_{n+1}-nP_n \}.
\end{displaymath}
If we introduce the generating function
\begin{displaymath}
G(t,x)=\sum_{n=0}^{\infty} P_n(t) x^n,
\end{displaymath}
it is easy to show that $G$ satisfies the equation
\begin{equation}\label{genvac}
\dfrac{\partial G}{\partial t}=\lambda (1-x)\dfrac{\partial G}{\partial x}.
\end{equation}
Solution of \eqref{genvac}, subject to the initial condition $G(0,x)=G_0(x)$ and the boundary condition $G(t,1)=1$ (that comes from the conservation
of the total probability), reads
\begin{equation*}
G(t,x)=G_0 \left( 1+ (x-1)e^{-\lambda t} \right).
\end{equation*}

On the other hand, the Poisson representation of the generating function is given by
\begin{equation*}
G(t,x)=\int_0^{\infty}\Psi(t,\phi)e^{\phi (x-1)}d\phi.
\end{equation*}
The corresponding equation of motion for the amplitude $\Psi$ reads
\begin{equation}\label{pde1}
\dfrac{\partial \Psi}{\partial t}=\lambda \dfrac{\partial}{\partial \phi}(\phi \Psi).
\end{equation}
This equation can be solved by the method of characteristics, which yields the ODE
\begin{equation}\label{detsde}
\frac{d \phi}{dt}=-\lambda \phi.
\end{equation}
Then the solution to equation~\eqref{pde1} reads
\begin{displaymath}
\Psi(t,\phi)=e^{\lambda t}\Psi_0(\phi e^{\lambda t}),
\end{displaymath}
and thus
\begin{equation}\nonumber
G(t,x)=\int_0^{\infty} e^{\lambda t}\Psi_0(\phi e^{\lambda t})e^{\phi (x-1)}d\phi,
\end{equation}
which yields
\begin{displaymath}
G(t,x)=G_0 \left( 1+ (x-1)e^{-\lambda t} \right),
\end{displaymath}
after taking into account that
\begin{displaymath}
G_0(x)=\int_0^{\infty}\Psi_0(\phi)e^{\phi (x-1)}d\phi.
\end{displaymath}

\begin{remark}
We  finish this subsection with two conclusions:
\begin{itemize}
\item The equivalence of equations~\eqref{genvac} and~\eqref{pde1} suggests the correctness of the procedure.
\item Equation~\eqref{detsde} plays the role of equation~\eqref{complex1} in the previous sections, but in this
case it has being well derived using the method of characteristics.
\end{itemize}
\end{remark}

\subsection{Real Noise}
Consider now the set of reactions
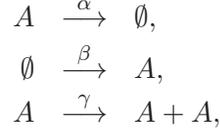
\begin{eqnarray} \nonumber
A &\stackrel{\alpha}{\longrightarrow}& \emptyset, \\ \nonumber
\emptyset &\stackrel{\beta}{\longrightarrow}& A, \\ \nonumber
A &\stackrel{\gamma}{\longrightarrow}& A + A,
\end{eqnarray}
which can be described via the forward Kolmogorov equation
$$
\frac{d P_n}{dt}= \gamma [(n-1)P_{n-1} -n P_n] + \beta(P_{n-1} - P_n) + \alpha [(n+1)P_{n+1} - nP_n].
$$
For the sake of analytical tractability we will make the choice $\alpha=\gamma=\beta$. Then we find the equation
$$
\alpha^{-1} \dfrac{\partial G}{\partial t} = (x-1)G + (x-1)^2\dfrac{\partial G}{\partial x} = (x-1) \dfrac{\partial}{\partial x} \left( (x-1) G \right),
$$
to be solved for the generating function $G$, together with initial condition $G(0,x) = G_0(x)$ and boundary condition $G(t,1)=1$. Solution of this problem can be computed with the method of characteristics and it reads
\begin{equation}\label{genmc2}
G(t,x)= \frac{1}{1- \alpha t (x-1)} \,\, G_0 \left( \frac{x- \alpha t (x-1)}{1- \alpha t (x-1)} \right).
\end{equation}
The amplitude $\Psi$ obeys the equation
\begin{eqnarray}\nonumber
\alpha^{-1} \partial_t \Psi &= \partial_{\phi}(\phi \partial_{\phi} \Psi) & = -\partial_{\phi} \Psi + \partial_{\phi}^2 (\phi \Psi),
\end{eqnarray}
which is the Fokker-Planck equation that corresponds to the SDE
$$
d\phi = \alpha \, dt + \sqrt{2 \alpha \phi}\,dW_t,
$$
the unique solution of which is a time-rescaled Squared Bessel process of dimension $\delta=2$ \cite{jean},
which implies, among other things, that its density $\Psi$ is smooth \cite{chen}.
From the coherent transform we can recover the generating function
\begin{equation}\label{gener_func_2}
G(t,x) = \int_{0}^{\infty} e^{\phi(x-1)}\Psi(t,\phi)d\phi.
\end{equation}
It is important to note that the differential operator $A(\cdot)=\partial_{\phi}[\phi \partial_{\phi}(\cdot)]$ is symmetric.
To see the importance of this fact define $\xi(t,\phi)$ to be the solution of the Cauchy problem
\begin{eqnarray*}
\alpha^{-1} \partial_t\xi(t,\phi)  & = & A(\xi) = \partial_{\phi}(\phi\partial_{\phi}\xi), \\
\xi(0,\phi) & = & e^{\phi(x-1)},
\end{eqnarray*}
which can be solved to yield
\begin{equation}\nonumber
\xi(t,\phi) = \dfrac{1}{1-\alpha t(x-1)}e^{\frac{\phi(x-1)}{1-\alpha t(x-1)}}.
\end{equation}
Now we claim that the integral
\begin{equation}\nonumber
I(t,s) = \int_{0}^{\infty} \xi(t-s,\phi) \Psi(s,\phi) d\phi
\end{equation}
is independent of $s$.

To see this, take the derivative of $I$ with respect to $s$ to obtain
\begin{eqnarray}\nonumber
\alpha^{-1}\dfrac{d}{ds}I(t,s) & = & - \int_{0}^{\infty} \alpha^{-1} \partial_t\xi(t-s,\phi) \Psi(s,\phi) d\phi + \int_{0}^{\infty} \alpha^{-1} \xi(t-s,\phi) \partial_s\Psi(s,\phi) d\phi \\ \nonumber
& = & - \int_{0}^{\infty} \left( A \xi \right)(t-s,\phi) \Psi(s,\phi) d\phi  + \int_{0}^{\infty} \xi(t-s,\phi) \left( A \Psi \right) (s,\phi) d\phi
\\ \nonumber & = & \int_{0}^{\infty} \xi(t-s,\phi) \left( A - A^{T} \right) \Psi(s,\phi) d\phi = 0.
\end{eqnarray}
In particular,
\begin{equation}\nonumber
\int_{0}^{\infty} e^{\phi(x-1)} \Psi(t,\phi) d\phi =  \int_{0}^{\infty}\xi(t,\phi) \Psi_0(\phi) d\phi.
\end{equation}
As a consequence we can compute, using \eqref{gener_func_2},
\begin{eqnarray*}
G(t,x) %&=& \int_{0}^{\infty} e^{\phi (x-1)} \Psi(t,\phi)d\phi \\
& = & \dfrac{1}{1-\alpha t(x-1)} \int_0^{\infty} e^{\frac{\phi(x-1)}{1-\alpha t(x-1)}}\Psi_0(\phi)d\phi  \\
&=& \dfrac{1}{1-\alpha t(x-1)} \int_0^{\infty} e^{\phi \left( \frac{x-\alpha t(x-1)}{1-\alpha t(x-1)} -1\right)}\Psi_0(\phi)d\phi \\ &=& \dfrac{1}{1-\alpha t(x-1)}G_0\left( \frac{x- \alpha t (x-1)}{1- \alpha t (x-1)} \right),
\end{eqnarray*}
which is in perfect agreement with~\eqref{genmc2}. Note that this last result makes sense even if $\Psi_0$ is not a probability
measure. This again  suggests two conclusions:
\begin{itemize}
\item  The procedure gives again correct results, but the SDE has been derived correctly within the framework
of stochastic analysis.
\item The correctness of the method even for $\Psi_0$ not being a probability measure suggests that the equation
for $\Psi$ is more general than the SDE.
\end{itemize}
A similar derivation, but formal from the viewpoint of our theory, is presented in Appendix~\ref{morern} to illustrate the robustness
of this type of computations.
%Therein we have to neglect a boundary term that arises upon integration by parts in the derivation of the equation for the amplitude in order to carry out our explicit calculation till the end.

\section{Imaginary Noise is not Unreal}\label{notunreal}

In this section we point out the fact that there might be a connection between the complex SDE~\eqref{isde} and the solution to~\eqref{master} deeper than the
already stated relation between the respective moments~\cite{gardiner,gardiner0,gardiner1}; see also Lemma~\ref{theo:moments} below.
A generalization of~\eqref{bad1} is the following one-dimensional negative-diffusion Fokker-Planck-like equation
\begin{equation} \nonumber
\dfrac{\partial \Psi}{\partial t}= -\dfrac{\partial}{\partial \phi}\Big( A(\phi) \Psi \Big)-\dfrac{1}{2}\dfrac{\partial^2}{\partial \phi^2}\Big( D^2(\phi) \Psi \Big),
\end{equation}
where $A(\phi)$ and $D(\phi)$ are polynomials. % w.r.t.~$\phi$.

As we have already seen, we have to expect a distribution to be the solution of such a partial differential equation.
In order to move forward, it will be more convenient to employ the complex analytic representation of distributions~\cite{debrou}.
Let us denote by $\mathcal{H} \left(\cdot\right)$ the class of holomorphic functions on a given domain, it can then be proved that, see Theorem 2.2.10 in \cite{debrou},

\begin{thm} [Analytic representation of distributions]
For every $\Psi \in C_c^{\infty}(\mathbb{R})'$ there exists a $\left\lbrace \Psi\right\rbrace_a \in \mathcal{H} \left( \mathbb{C}\backslash \mathbb{R} \right)$
such that for all $f \in C_c^{\infty}(\mathbb{R})$,
	\begin{displaymath}
	\left< \Psi | f \right> = \lim\limits_{\phi_2 \rightarrow 0^+ }  \int_{-\infty}^{\infty} \Big(\left\lbrace \Psi\right\rbrace_a (\phi_1 +i \phi_2)- \left\lbrace \Psi\right\rbrace_a(\phi_1 -i \phi_2) \Big) f(\phi_1) d\phi_1,
	\end{displaymath}
where $\left< \cdot | \cdot \right>$ represents the duality product between $C_c^{\infty}(\mathbb{R})$
and $C_c^{\infty}(\mathbb{R})'$.
\end{thm}

Note that this representation is not unique as any $\left\lbrace \Psi\right\rbrace_a \in \mathcal{H} \left( \mathbb{C} \right)$ leads to the trivial distribution, see e.g.~\cite{kothe}.
We focus now on the following analytical representation.

\begin{defi}[Cauchy representation]\label{cauchyr}
For every $\Psi \in C^{\infty}(\mathbb{R})'$ we define its Cauchy representation as
	\begin{displaymath}
	\left\lbrace \Psi\right\rbrace_a(\phi):=\dfrac{1}{2\pi i}\left< \Psi(s) \Bigg| \dfrac{1}{s-\phi} \right>_s,
	\end{displaymath}
	where $\phi \in \mathbb{C}\backslash \mathbb{R}$.
\end{defi}

\begin{remark}
It is possible to prove that $\left\lbrace \Psi\right\rbrace_a(\phi)$ is always well-defined in $\mathcal{H} \left( \mathbb{C}\backslash \mathbb{R} \right)$~\cite{debrou}.
Whenever $\Psi \in C_c(\mathbb{R})$ its Cauchy representation can be written as the integral
\begin{displaymath}
\left\lbrace \Psi\right\rbrace_a(\phi)=\dfrac{1}{2\pi i}\int_{-\infty}^{+\infty} \dfrac{\Psi(s)}{s-\phi}ds.
\end{displaymath}
Clearly, when $\Psi \in C^\infty \left( \mathbb{R} \right)' \backslash C_c(\mathbb{R})$,
the duality product in Definition~\ref{cauchyr} is well defined, but the integral is not necessarily so.
\end{remark}

A paradigmatic example of Cauchy representation is that of the $n-$th derivative of the Dirac delta:
$$
\mathrm{If} \quad \Psi=\delta_0^{(n)} \quad \mathrm{then} \quad \left\lbrace \Psi\right\rbrace_a(\phi)=\dfrac{1}{2 \pi i }\dfrac{(-1)^{n+1} n!}{\phi^{n+1}}.
$$

\begin{prop}\label{propcr}
Let $\{\Psi\}_a \in \mathcal{H} \left( \mathbb{C}\backslash \mathbb{R} \right)$ be the Cauchy representation of a
(compactly-supported) distribution $\Psi \in C^\infty \left( \mathbb{R} \right)'$. Then:
	\begin{itemize}
		\item $\left\{\mathcal{P}(s) \Psi(s) \right\} _a= \mathcal{P}(\phi) \{\Psi\}_a(\phi)$ in
$\mathcal{H} \left( \mathbb{C}\backslash \mathbb{R} \right)/\mathcal{H} \left( \mathbb{C}\right)$ and
		\item $\left\{ \dfrac{\partial^m \Psi(s)}{\partial s^m}  \right\}_a=\dfrac{\partial^m \{\Psi(\phi)\}_a}{\partial \phi^m}$ in
$\mathcal{H} \left( \mathbb{C}\backslash \mathbb{R} \right)/\mathcal{H} \left( \mathbb{C}\right)$,
	\end{itemize}
where $\phi \in \mathbb{C}\backslash \mathbb{R}$, $s \in \mathbb{R}$, $\mathcal{P}$ is an arbitrary polynomial, and $m$ an arbitrary positive integer.
\end{prop}

\begin{proof}
It is clear that all expressions in the statement are well-defined. Now, to prove the first property  it is enough to show that the equality is
true for $\mathcal{P}$ being an arbitrary monomial, say $\phi^n$. The case $n=0$ is trivial, for $n=1$ compute
		\begin{eqnarray}\nonumber
		\phi \{\Psi\}_a(\phi) &=& \dfrac{1}{2\pi i}\left<  \Psi(s) \Bigg| \dfrac{\phi}{s-\phi} \right>_s = \dfrac{1}{2\pi i} \left< \Psi(s) \Bigg| \dfrac{s - (s - \phi)}{s-\phi} \right>_s \\ \nonumber
		&=& \dfrac{1}{2\pi i} \left< s\Psi(s) \Bigg| \dfrac{1}{s-\phi} \right>_s - \dfrac{1}{2\pi i}\left< \Psi(s) \Bigg| 1 \right>_s\\ \nonumber
		&=& \{s \Psi(s)\}_a-\underbrace{\dfrac{1}{2\pi i} \left< \Psi(s) \Bigg| 1 \right>_s}_{\in \mathcal{H} \left( \mathbb{C}\right)}.
		\end{eqnarray}
The case $n>1$ follows from the following computations:
        \begin{eqnarray}\nonumber
        \phi^n \{\Psi\}_a(\phi) &=& \dfrac{1}{2\pi i}\left<  \Psi(s) \Bigg| \dfrac{\phi^n}{s-\phi} \right>_s = \dfrac{1}{2\pi i}\left<  \Psi(s) \Bigg| \dfrac{s^n - (s^n -\phi^n)}{s-\phi} \right>_s \\ \nonumber
        &=& \dfrac{1}{2\pi i}\left<  s^n\Psi(s) \Bigg| \dfrac{1}{s-\phi} \right>_s
        -\dfrac{1}{2\pi i} \left< \Psi(s) \Bigg| \dfrac{s^n - \phi^n}{s-\phi} \right>_s
        \\ \nonumber
        &=& \{s^n \Psi(s)\}_a
        -\underbrace{\dfrac{1}{2\pi i} \sum_{m=0}^{n-1} \phi^{n-1-m} \left<  \Psi(s) \Bigg| s^m \right>_s }_{\in \mathcal{H} \left( \mathbb{C}\right)}.
        \end{eqnarray}
To prove the second property we proceed by induction, commencing with the case $m=1$:
		\begin{eqnarray}\nonumber
		\dfrac{\partial}{\partial \phi} \{\Psi\}_a(\phi) &=&
        \dfrac{1}{2\pi i} \left<  \Psi(s) \Bigg| \dfrac{ \partial }{\partial \phi}\dfrac{1}{s-\phi} \right>_s =
        \dfrac{-1}{2\pi i} \left<  \Psi(s) \Bigg| \dfrac{ \partial }{\partial s}\dfrac{1}{s-\phi} \right>_s \\ \nonumber
        &=& \dfrac{1}{2\pi i}\left<  \Psi'(s) \Bigg| \dfrac{1}{s-\phi} \right>_s = \left\{\frac{\partial \Psi(s)}{\partial s} \right\}_a.
		\end{eqnarray}
For the general case we use the induction hypothesis to find
\begin{eqnarray}\nonumber
		\dfrac{\partial^{m+1}}{\partial \phi^{m+1}} \{\Psi\}_a(\phi) &=&
        \dfrac{1}{2\pi i}\left<  \dfrac{\partial^{m} \Psi(s)}{\partial s^{m}} \Bigg| \dfrac{\partial }{\partial \phi}\dfrac{1}{s-\phi} \right>_s
        = \dfrac{-1}{2\pi i}\left<  \dfrac{\partial^{m} \Psi(s)}{\partial s^{m}} \Bigg| \dfrac{ \partial }{\partial s}\dfrac{1}{s-\phi} \right>_s
        \\ \nonumber
        &=& \dfrac{1}{2\pi i}\left<  \dfrac{\partial^{m+1} \Psi(s)}{\partial s^{m+1}} \Bigg| \dfrac{1}{s-\phi} \right>_s
        = \left\{\dfrac{\partial^{m+1} \Psi(s)}{\partial s^{m+1}} \right\}_a.
\end{eqnarray}
\end{proof}

From now one we will always assume the initial condition $\Psi_0 \in C^\infty(\mathbb{R})'$.

\begin{cor}
Let $A(\cdot)$ and $D(\cdot)$ be polynomials and let $\Psi$ be a $C^1([0,T],C^\infty(\mathbb{R})')$ solution, for some $T>0$, to
\begin{equation}\label{distpde}
\begin{cases}
\dfrac{\partial \Psi}{\partial t} = -\dfrac{\partial}{\partial \phi}\left[ A(\phi) \Psi \right]-\dfrac{1}{2}\dfrac{\partial^2}{\partial \phi^2}\left[ D^2(\phi) \Psi \right], \\
\Psi(0,\phi) = \Psi_0(\phi), \quad \phi \in \mathbb{R},
\end{cases}
\end{equation}
then its Cauchy representation $\{\Psi\}_a$ is a
$C^1([0,T],\mathcal{H} \left( \mathbb{C}\backslash \mathbb{R} \right)/\mathcal{H} \left( \mathbb{C}\right))$
solution to
\begin{equation}\label{distpdecomplex0}
\begin{cases}
\dfrac{\partial \{\Psi\}_a}{\partial t}= -\dfrac{\partial}{\partial \phi}\left[ A(\phi) \{\Psi\}_a \right]-\dfrac{1}{2}\dfrac{\partial^2}{\partial \phi^2}
\left[ D^2(\phi) \{\Psi\}_a \right], \\
\{\Psi\}_a(0,\phi)=\{\Psi_0\}_a(\phi), \quad \phi \in \mathbb{C}.
\end{cases}
\end{equation}
\end{cor}
\begin{proof}
The statement is a direct consequence of Proposition~\ref{propcr}.
\end{proof}

We have already connected distribution-valued solutions to the negative-diffusion PDE \eqref{distpde}
% 	\begin{equation}\label{distpde}
% 	\begin{cases}
% 	\dfrac{\partial \Psi}{\partial t}= -\dfrac{\partial}{\partial \phi}\left\lbrace A(\phi) \Psi \right\rbrace-\dfrac{1}{2}\dfrac{\partial^2}{\partial \phi^2}\left\lbrace D^2(\phi) \Psi \right\rbrace, \\
% 	\Psi(\phi,0)=\Psi_0(\phi), \quad \phi \in \mathbb{R},
% 	\end{cases}
% 	\end{equation}
with solutions to the complex PDE \eqref{distpdecomplex0}
% 	\begin{equation}\label{distpdecomplex0}
% 	\begin{cases}
% 	\dfrac{\partial \{\Psi\}_a}{\partial t}= -\dfrac{\partial}{\partial \phi}\left\lbrace A(\phi) \{\Psi\}_a \right\rbrace-\dfrac{1}{2}\dfrac{\partial^2}{\partial \phi^2}\left\lbrace D^2(\phi) \{\Psi\}_a \right\rbrace, \\
% 	\{\Psi\}_a(\phi,0)=\{\Psi_0\}_a(\phi), \quad \phi \in \mathbb{C},
% 	\end{cases}
% 	\end{equation}
via the Cauchy representation $\{\cdot\}_a$. The imaginary-noise SDE formally associated to problem~\eqref{distpdecomplex0} is
	\begin{equation}\label{complexsdedist}
	dz=A(z)dt+iD(z)dW_t,
	\end{equation}
	where $z \in \mathbb{C}$. And this SDE is in turn associated with the real two-dimensional Fokker-Planck equation
	\begin{equation}\label{realfokker0}
	\dfrac{\partial }{\partial t} P=-\dfrac{\partial}{\partial z_1}\left[ A_1 P\right] - \dfrac{\partial}{\partial z_2}\left[ A_2 P\right] +\dfrac{1}{2}\dfrac{\partial^2}{\partial z_1^2}\left[ D_2^2 P\right] +\dfrac{1}{2}\dfrac{\partial^2}{\partial z_2^2}\left[ D_1^2 P\right]
-\dfrac{\partial^2}{\partial z_1\partial z_2 }\left[ D_1 D_2 P\right],
	\end{equation}
where $A_1(z_1,z_2)=\Re[A(z_1 + i z_2)]$, $A_2(z_1,z_2)=\Im[A(z_1 + i z_2)]$,
$D_1(z_1,z_2)=\Re[D(z_1 + i z_2)]$, and $D_2(z_1,z_2)=\Im[D(z_1 + i z_2)]$.
The following result shows how to connect~\eqref{distpde} with~\eqref{complexsdedist} through~\eqref{distpdecomplex0} and~\eqref{realfokker0}.

\begin{thm} \label{connection}
	Let $\Psi$ be a $C^1([0,T],C^\infty(\mathbb{R})')$ solution to
	\eqref{distpde} .
% 	\begin{equation}\nonumber
% 	\begin{cases}
% 	\dfrac{\partial \Psi}{\partial t}= -\dfrac{\partial}{\partial \phi}\left[ A(\phi) \Psi \right]-\dfrac{1}{2}\dfrac{\partial^2}
%     {\partial \phi^2}\left[ D^2(\phi) \Psi \right] \\
% 	\Psi(\phi,0)=\Psi_0(\phi), \quad \phi \in \mathbb{R}.
% 	\end{cases}
% 	\end{equation}
Then its Cauchy representation can be expressed as
	\begin{equation}\nonumber
	\{\Psi\}_a(t, \phi)=\dfrac{1}{2\pi i} \left< \Psi_0(x_0) \Bigg| \int_{-\infty}^{+\infty} \int_{-\infty}^{+\infty}
	\dfrac{P(t, z_1,z_2)}{z_1 + i z_2-\phi}
	dz_1 dz_2\right>_{x_0}
	\end{equation}
for $\phi \in \mathbb C \setminus \overline{\text{supp}(P) \cup \mathbb{R}}$, if there exists a unique compactly-supported solution $P \in C^1([0,T],\mathcal{M}(\mathbb{R}^2))$  to  \eqref{realfokker0} with initial condition
$$
P(0, z_1,z_2)=\delta(z_1-x_0) \delta(z_2),
$$
% 	\begin{equation}\nonumber
% 	\begin{cases}
% 	\dfrac{\partial }{\partial t} P=-\dfrac{\partial}{\partial z_1}\left[ A_1 P\right] - \dfrac{\partial}{\partial z_2}\left[ A_2
%     P\right] +\dfrac{1}{2}\dfrac{\partial^2}{\partial z_1^2}\left[ D_2^2 P\right] + \dfrac{1}{2}\dfrac{\partial^2}{\partial z_2^2}\left[ D_1^2 P\right] -\dfrac{\partial^2}{\partial z_1\partial z_2 }\left[ D_1 D_2 P\right] \\
% 	P(z_1,z_2,0)=\delta(z_1-x_0) \delta(z_2),
% 	\end{cases}
% 	\end{equation}
where
%$A_1(z_1,z_2)=\Re[A(z_1 + i z_2)]$, $A_2(z_1,z_2)=\Im[A(z_1 + i z_2)]$,$D_1(z_1,z_2)=\Re[D(z_1 + i z_2)]$, $D_2(z_1,z_2)=\Im[D(z_1 + i z_2)]$, and
 $\mathcal{M}(\mathbb{R}^2)$
is the space of all probability measures over $\mathbb{R}^2$.
\end{thm}

\begin{proof}
We start with the initial condition:
    \begin{eqnarray}\nonumber
	\left\lbrace \Psi\right\rbrace_a(0,\phi)  &=& \dfrac{1}{2\pi i}\left< \Psi_0(x_0) \Bigg| \int_{-\infty}^{+\infty} \int_{-\infty}^{+\infty}
	\dfrac{P(0,z_1,z_2)}{z_1 + i z_2-\phi}
	dz_1 dz_2\right>_{x_0} \\ \nonumber
    &=& \dfrac{1}{2\pi i}\left< \Psi_0(x_0) \Bigg| \int_{-\infty}^{+\infty} \int_{-\infty}^{+\infty}
    \dfrac{\delta(z_1-x_0) \delta(z_2)}{z_1 + i z_2-\phi}
    dz_1 dz_2\right>_{x_0} \\ \nonumber
    &=& \dfrac{1}{2\pi i} \left< \Psi_0(x_0) \Bigg|
    \dfrac{1}{x_0-\phi} \right>_{x_0}.
	\end{eqnarray}
Using the relations
	\begin{displaymath}
	\dfrac{\partial}{\partial z_1} \dfrac{1}{z_1 + i z_2 -\phi}=-\dfrac{\partial}{\partial \phi}\dfrac{1}{z_1 + i z_2 -\phi}
	\end{displaymath}
	and
	\begin{displaymath}
	i \dfrac{\partial}{\partial z_2} \dfrac{1}{z_1 + i z_2 -\phi}=\dfrac{\partial}{\partial \phi}\dfrac{1}{z_1 + i z_2 -\phi},
	\end{displaymath}	
together with equation \eqref{realfokker0} for $P$, we find
    \begin{eqnarray}\nonumber
	\dfrac{\partial}{\partial t}\left\lbrace \Psi\right\rbrace_a(t,\phi) &=&
    \dfrac{1}{2\pi i}\left< \Psi_0(x_0) \Bigg| \int_{-\infty}^{+\infty} \int_{-\infty}^{+\infty}
    \dfrac{\partial_t P(t,z_1,z_2)}{z_1 + i z_2-\phi}dz_1 dz_2\right>_{x_0}
%     \\ \nonumber &=&
%     -\dfrac{1}{2\pi i}\left< \Psi_0(x_0) \Bigg| \int_{-\infty}^{+\infty} \int_{-\infty}^{+\infty}
%     \dfrac{\partial_{z_1} ( A_1 P(z_1,z_2,t))}{z_1 + i z_2-\phi}dz_1 dz_2\right>_{x_0}
%     \\ \nonumber & &
%     -\dfrac{1}{2\pi i}\left< \Psi_0(x_0) \Bigg| \int_{-\infty}^{+\infty} \int_{-\infty}^{+\infty}
%     \dfrac{\partial_{z_2} ( A_2 P(z_1,z_2,t))}{z_1 + i z_2-\phi}dz_1 dz_2\right>_{x_0}
% 	\\ \nonumber & &
%     +\dfrac{1}{4\pi i}\left< \Psi_0(x_0) \Bigg| \int_{-\infty}^{+\infty} \int_{-\infty}^{+\infty}
%     \dfrac{\partial_{z_1}^2 ( D_2^2 P(z_1,z_2,t))}{z_1 + i z_2-\phi}dz_1 dz_2\right>_{x_0}
%     \\ \nonumber & &
%     +\dfrac{1}{4\pi i}\left< \Psi_0(x_0) \Bigg| \int_{-\infty}^{+\infty} \int_{-\infty}^{+\infty}
%     \dfrac{\partial_{z_2}^2 ( D_1^2 P(z_1,z_2,t))}{z_1 + i z_2-\phi}dz_1 dz_2\right>_{x_0}
%     \\ \nonumber & &
%     -\dfrac{1}{2\pi i}\left< \Psi_0(x_0) \Bigg| \int_{-\infty}^{+\infty} \int_{-\infty}^{+\infty}
%     \dfrac{\partial_{z_1}\partial_{z_2} ( D_1D_2 P(z_1,z_2,t))}{z_1 + i z_2-\phi}dz_1 dz_2\right>_{x_0}
     \\ \nonumber &=&
     -\dfrac{1}{2\pi i}\left< \Psi_0(x_0) \Bigg| \int_{-\infty}^{+\infty} \int_{-\infty}^{+\infty}
     \dfrac{\partial}{\partial \phi} \dfrac{A_1 P(t,z_1,z_2)}{z_1 + i z_2-\phi}dz_1 dz_2\right>_{x_0}
    \\ \nonumber & &
    -\dfrac{1}{2\pi i}\left< \Psi_0(x_0) \Bigg| \int_{-\infty}^{+\infty} \int_{-\infty}^{+\infty}
    \dfrac{\partial}{\partial \phi}\dfrac{ iA_2 P(t, z_1,z_2)}{z_1 + i z_2-\phi}dz_1 dz_2\right>_{x_0}
	\\ \nonumber & &
    +\dfrac{1}{4\pi i}\left< \Psi_0(x_0) \Bigg| \int_{-\infty}^{+\infty} \int_{-\infty}^{+\infty}
    \dfrac{\partial^2}{\partial \phi^2} \dfrac{ D_2^2 P(t,z_1,z_2)}{z_1 + i z_2-\phi}dz_1 dz_2\right>_{x_0}
    \\ \nonumber & &
    +\dfrac{1}{4\pi i}\left< \Psi_0(x_0) \Bigg| \int_{-\infty}^{+\infty} \int_{-\infty}^{+\infty}
    -\dfrac{\partial^2}{\partial \phi^2}\dfrac{ D_1^2 P(t,z_1,z_2)}{z_1 + i z_2-\phi}dz_1 dz_2\right>_{x_0}
    \\ \nonumber & &
    -\dfrac{1}{2\pi i}\left< \Psi_0(x_0) \Bigg| \int_{-\infty}^{+\infty} \int_{-\infty}^{+\infty}
    \dfrac{\partial^2}{\partial \phi^2}\dfrac{i D_1D_2 P(t,z_1,z_2)}{z_1 + i z_2-\phi}dz_1 dz_2\right>_{x_0}
    \\ \nonumber &=&
    -\dfrac{1}{2\pi i}\left< \Psi_0(x_0) \Bigg| \int_{-\infty}^{+\infty} \int_{-\infty}^{+\infty}
    \dfrac{\partial}{\partial \phi}\dfrac{A(z_1,z_2) P(t,z_1,z_2)}{z_1 + i z_2-\phi}dz_1 dz_2\right>_{x_0}
    \\ \nonumber & &
    -\dfrac{1}{4\pi i}\left< \Psi_0(x_0) \Bigg| \int_{-\infty}^{+\infty} \int_{-\infty}^{+\infty}
    \dfrac{\partial^2}{\partial \phi^2}\dfrac{D^2(z_1,z_2) P(t,z_1,z_2)}{z_1 + i z_2-\phi}dz_1 dz_2\right>_{x_0}
    \\ \nonumber &=&
    -\dfrac{\partial}{\partial \phi} \dfrac{1}{2\pi i}\left< \Psi_0(x_0) \Bigg| \int_{-\infty}^{+\infty} \int_{-\infty}^{+\infty}
    \dfrac{A(z_1,z_2) P(t,z_1,z_2)}{z_1 + i z_2-\phi}dz_1 dz_2\right>_{x_0}
    \\ \nonumber & &
    -\dfrac{1}{2}\dfrac{\partial^2}{\partial \phi^2}\dfrac{1}{2\pi i}\left< \Psi_0(x_0) \Bigg| \int_{-\infty}^{+\infty} \int_{-\infty}^{+\infty} \dfrac{D^2(z_1,z_2) P(t,z_1,z_2)}{z_1 + i z_2-\phi}dz_1 dz_2\right>_{x_0}.
\end{eqnarray}
This implies the result stated in the theorem, since, by the argument similar to that in the proof of Proposition~\ref{propcr}, we have
\begin{eqnarray} \nonumber & &
    \dfrac{1}{2\pi i}\left< \Psi_0(x_0) \Bigg| \int_{-\infty}^{+\infty} \int_{-\infty}^{+\infty}
    \dfrac{A(z_1,z_2) P(t,z_1,z_2)}{z_1 + i z_2-\phi}dz_1 dz_2\right>_{x_0}
    \\ \nonumber && - \dfrac{A}{2\pi i}\left< \Psi_0(x_0) \Bigg| \int_{-\infty}^{+\infty} \int_{-\infty}^{+\infty}
    \dfrac{P(t,z_1,z_2)}{z_1 + i z_2-\phi}dz_1 dz_2\right>_{x_0} \in \mathcal{H}(\mathbb{C})
\end{eqnarray}
and
\begin{eqnarray} \nonumber & &
    \dfrac{1}{2\pi i}\left< \Psi_0(x_0) \Bigg| \int_{-\infty}^{+\infty} \int_{-\infty}^{+\infty} \dfrac{D^2(z_1,z_2) P(t,z_1,z_2)}{z_1 + i z_2-\phi}dz_1 dz_2\right>_{x_0}
    \\ \nonumber && -
    \dfrac{D^2}{2\pi i}\left< \Psi_0(x_0) \Bigg| \int_{-\infty}^{+\infty} \int_{-\infty}^{+\infty} \dfrac{P(t,z_1,z_2)}{z_1 + i z_2-\phi}dz_1 dz_2\right>_{x_0}
    \in \mathcal{H}(\mathbb{C}).
\end{eqnarray}
\end{proof}

\begin{remark}
Notice that solutions of equation~\eqref{realfokker0} corresponding to \eqref{isde}, with $A(z)= - z^2$ and $D(z) = z$, are compactly supported,
at least for some periods of time (i.e.~for $T$ sufficiently small), as shown in section~\ref{dynamics}.
\end{remark}

\subsection{Compactly Supported Initial Conditions}

In this subsection, unless explicitly indicated, we restrict ourselves to an important particular case:
compactly supported initial conditions for the Markov chain,
i.e.~we assume that $P_n(0)=0$ for all $n > N$,
where $N \in \mathbb{N}$ is arbitrarily large but fixed. We need the following preparatory results.

\begin{lem}
	The factorial moments
\begin{equation}\nonumber
\mathcal{M}_m(t):=\mathbb{E}\left[n(n-1)\cdots(n-m+1)\right](t), \,\, \forall \, m\in \mathbb{N}\cup\left\lbrace 0\right\rbrace
\end{equation}
    where $\mathbb{E}\left[\cdot\right]:=\sum_n (\cdot) P_n$,
    fulfil the system of coupled differential equations
	\begin{equation}\label{eq:momentssystem}
	\dfrac{d}{dt}\mathcal{M}_m=-\dfrac{m(m-1)}{2}\mathcal{M}_m-m\mathcal{M}_{m+1}, \,\, \forall \, m\in \mathbb{N}\cup\left\lbrace 0\right\rbrace.
	\end{equation}
\end{lem}

\begin{proof}
First we set $\lambda=1$ without loss of generality.
The factorial moments can be computed as derivatives of the generating function
\begin{equation}\nonumber
\mathcal{M}_m(t)=\dfrac{\partial^m G}{\partial x^m}(t,1).
\end{equation}
Now, by taking $m$ derivatives with respect to $x$ in equation \eqref{gfpde} we find
\begin{displaymath}
\dfrac{\partial}{\partial t}\dfrac{\partial^m G}{\partial x^m}=\dfrac{1}{2}\sum_{j=0}^{m}\binom{m}{j}\dfrac{\partial^{m-j}}{\partial x^{m-j}}(1-x^2)\dfrac{\partial^{j+2}G}{\partial x^{j+2}},
\end{displaymath}
and by evaluating the last expression at $x=1$ the statement follows.
\end{proof}

\begin{remark}
Note that this result is valid for any initial condition independently of the fact that it is compactly supported or not. For the compactly supported  initial condition  system \eqref{eq:momentssystem} is finite-dimensional, since
$\mathcal{M}_m(t)=0$ for all $m>N$. This fact is crucial for the following result, which is indeed restricted to that case.
\end{remark}

\begin{lem}\label{theo:moments}
The moments of the SDE~\eqref{isde} coincide identically with the factorial moments $\mathcal{M}_m(t)$ as long as its solution exists.
\end{lem}	

\begin{proof}
Applying It\^o formula to a entire function $f(\cdot)$ of the solution to SDE \eqref{isde} yields
	\begin{equation}\nonumber
	df(\phi)=-\phi^2 \left(  \dfrac{\partial f}{\partial \phi}+\dfrac{1}{2}\dfrac{\partial^2f}{\partial \phi^2}\right) dt+i\phi \dfrac{\partial f}{\partial \phi}dW;
	\end{equation}
the validity of this formula is proven in Appendix~\ref{sectvs}.
	Substituting $f(\phi)=\phi^m$ and applying the martingale property of the It\^o integral gives
	\begin{equation}\label{momentseq}
	\dfrac{d}{dt} \, \mathbb{E}(\phi^m) = -\dfrac{m(m-1)}{2} \, \mathbb{E}(\phi^m)-m \, \mathbb{E}(\phi^{m+1}),
	\end{equation}
	so this system is identical to \eqref{eq:momentssystem}. The statement follows from the classical uniqueness theorem for systems of
ordinary differential equations~\cite{arnold}.
\end{proof}

We now formulate the analog of Theorem~\ref{connection} in the present context.

\begin{thm} \label{connection2}
	Let $\Psi \in C^1([0,T],C^\infty(\mathbb{R})')$ be a solution to
	\eqref{distpde},
% 	\begin{equation}\nonumber
% 	\begin{cases}
% 	\dfrac{\partial \Psi}{\partial t}= -\dfrac{\partial}{\partial \phi}\left[ A(\phi) \Psi \right]-\dfrac{1}{2}\dfrac{\partial^2}
%     {\partial \phi^2}\left[ D^2(\phi) \Psi \right], \\
% 	\Psi(\phi,0)=\Psi_0(\phi), \quad \phi \in \mathbb{R},
% 	\end{cases}
% 	\end{equation}
 where $A(\phi)$ and $D(\phi)$ are polynomials. % w.r.t. $\phi$.

Then its Cauchy representation can be expressed as
	\begin{equation}\nonumber
	\{\Psi\}_a(t,\phi)=\dfrac{1}{2 \pi i} \sum_{n=0}^{\infty}\dfrac{(-1)^{n+1}}{\phi^{n+1}}
    \left< \Psi_0(x_0) \Bigg| \int_{-\infty}^{+\infty} \int_{-\infty}^{+\infty} (z_1 + i z_2)^n P(t,z_1,z_2)dz_1 dz_2\right>_{x_0}
	\end{equation}
for $\phi \in \mathbb{C} \setminus \{0\}$,
if there exists a unique solution $P \in C^1([0,T],\mathcal{M}(\mathbb{R}^2))$  to  \eqref{realfokker0} with initial condition
$$
P(0,z_1,z_2)=\delta(z_1-x_0) \delta(z_2),
$$
%
% \begin{equation}\nonumber
% \begin{cases}
% \dfrac{\partial }{\partial t} P=-\dfrac{\partial}{\partial z_1}\left[ A_1 P\right] - \dfrac{\partial}{\partial z_2}\left[ A_2
% P\right] +\dfrac{1}{2}\dfrac{\partial^2}{\partial z_1^2}\left[ D_2^2 P\right] + \dfrac{1}{2}\dfrac{\partial^2}{\partial z_2^2}\left[ D_1^2 P\right] -\dfrac{\partial^2}{\partial z_1\partial z_2 }\left[ D_1 D_2 P\right] \\
% P(z_1,z_2,0)=\delta(z_1-x_0) \delta(z_2),
% \end{cases}
% \end{equation}
with all moments finite.
%,where $A_1(z_1,z_2)=\Re[A(z_1 + i z_2)]$, $A_2(z_1,z_2)=\Im[A(z_1 + i z_2)]$, $D_1(z_1,z_2)=\Re[D(z_1 + i z_2)]$, and $D_2(z_1,z_2)=\Im[D(z_1 + i z_2)]$.
\end{thm}

\begin{remark}
The finiteness of all moments can be obviously replaced by the finiteness of the moment of the same order as the highest non-vanishing moment of
the initial distribution.
\end{remark}

\begin{proof}
First note that
\begin{eqnarray}\nonumber & &
\dfrac{1}{2 \pi i} \sum_{n=0}^{\infty}\dfrac{(-1)^{n+1}}{\phi^{n+1}}
\left< \Psi_0(x_0) \Bigg| \int_{-\infty}^{+\infty} \int_{-\infty}^{+\infty} (z_1 + i z_2)^n P(0,z_1,z_2)dz_1 dz_2\right>_{x_0}
\\ \nonumber &=&
\dfrac{1}{2 \pi i} \sum_{n=0}^{\infty}\dfrac{(-1)^{n+1}}{\phi^{n+1}}
\left< \Psi_0(x_0) \Bigg| \int_{-\infty}^{+\infty} \int_{-\infty}^{+\infty} (z_1 + i z_2)^n \delta(z_1-x_0) \delta(z_2) dz_1 dz_2\right>_{x_0}
\\ \nonumber &=&
\dfrac{1}{2 \pi i} \sum_{n=0}^{\infty}\dfrac{(-1)^{n+1}}{\phi^{n+1}}
\left< \Psi_0(x_0) \Bigg| x_0^n\right>_{x_0}
\\ \nonumber &=&
\dfrac{1}{2 \pi i} \sum_{n=0}^{\infty}\dfrac{(-1)^{n+1}}{\phi^{n+1}} \mathcal{M}_n(0),
\end{eqnarray}
where $\mathcal{M}_n(0)$ denote the moments of the distribution $\Psi_0(x_0)$.
Consider now the (infinite-dimensional) vector
\begin{eqnarray}\nonumber & &
\left\{
\left< \Psi_0(x_0) \Bigg| \int_{-\infty}^{+\infty} \int_{-\infty}^{+\infty} (z_1 + i z_2)^n P(t,z_1,z_2)dz_1 dz_2\right>_{x_0}
\right\}_{n=1}^\infty
\\ \nonumber &=&
\left\{ \mathcal{D}_{nn} \right\}_{\mathbb{Z} \times \mathbb{Z}} \left\{
\left< \Psi_0(x_0) \Bigg| \int_{-\infty}^{+\infty} \int_{-\infty}^{+\infty} (z_1 + i z_2)^n P(0,z_1,z_2)dz_1 dz_2\right>_{x_0}
\right\}_{n=1}^\infty
\\ \nonumber &=&
\left\{ \mathcal{D}_{nn} \right\}_{\mathbb{Z} \times \mathbb{Z}}
\left\{ \left< \Psi_0(x_0) \Bigg| x_0^n\right>_{x_0} \right\}_{n=1}^\infty
\\ \nonumber &=&
\mathcal{M}_n(t),
\end{eqnarray}
where $\left\{ \mathcal{D}_{nn} \right\}_{\mathbb{Z} \times \mathbb{Z}}$ denotes the infinite-dimensional operator that describes the time
evolution of the moments~\eqref{momentseq}.
Therefore
	\begin{equation}\nonumber
	\{\Psi\}_a(t,\phi)=\dfrac{1}{2 \pi i} \sum_{n=0}^{\infty}\dfrac{(-1)^{n+1}}{\phi^{n+1}} \mathcal{M}_n(t),
    \end{equation}
for all $t \ge 0$, and thus, as a distribution
\begin{equation}\nonumber
\Psi(t,\phi)=\sum_{n=0}^{\infty}\dfrac{\mathcal{M}_n(t)}{n!} \, \delta^{(n)}(\phi),
\end{equation}
which is our candidate for solution of the PDE in the statement.
Now we may, by means of the distributional version of~(\ref{G_2}), conclude by computing
\begin{eqnarray}\nonumber
G(t,x) &=& \Big< \Psi(t,\phi) \Big| e^{\phi (x-1)} \Big>_{\phi}  = \sum_{n=0}^{\infty}\dfrac{\mathcal{M}_n(t)}{n!} \, (1-x)^n,
\end{eqnarray}
so we recover our original generating function Taylor-expanded at $x=1$ instead of $x=0$
(note that, in the present case of compactly supported initial conditions, the generating function is simply a polynomial).
\end{proof}

\begin{remark}
The formal connection between Theorems~\ref{connection} and~\ref{connection2} comes from the identity
\begin{equation}\nonumber
\dfrac{1}{z-\phi} = \sum_{n=0}^{\infty}\dfrac{(-1)^{n+1}}{\phi^{n+1}} \, z^n,
\end{equation}
which is valid for $\left\lvert z\right\rvert < \left\lvert \phi\right\rvert$.
\end{remark}

\begin{remark}
The finiteness of the moments of the solution required in the statement of this theorem replaces the compact support condition in Theorem~\ref{connection}.
Actually both hold (note that the latter implies the former) for arbitrary periods of time, as shown in the next section. We will address this question simultaneously to that of the local in time existence of solutions to SDE~\eqref{isde}, which was required in Lemma~\ref{theo:moments}.
\end{remark}

\section{Stochastic Dynamics}
\label{dynamics}

In this section we analyze some dynamical features of equation~\eqref{isde}.
Obviously, one of the missing steps in building a complete theory of the imaginary It\^o interpretation is the global existence
of the solution to this equation. Although we do not offer a proof of this fact herein, we build some partial progress on it,
and we as well prove some characteristic dynamical features of this complex SDE. In particular, we show that its probability
density is not supported in all of $\mathbb{C}$, a necessary requirement that appears in the statement of Theorem~\ref{connection}.
With respect to Theorem~\ref{connection2}, our result on the local in time existence of the solution will be enough, as we will highlight
in the following.

Clearly $\phi=0$ is the unique absorbing state for this diffusion,
however it is an unstable state (perhaps contrary to intuition after regarding the development in Appendix~\ref{sectvan}).
We start with the precise statement of this fact.

\begin{thm}\label{thdynamic1}
Let $\phi(t)$ be a solution of \eqref{isde}; with probability one it holds that:
\begin{itemize}
\item if $\vert \phi(0)\vert >0$ then the solution exists at least for a positive interval of time
(if $\vert \phi(0)\vert =0$ then the solution trivially exists for all times);
\item if $\vert \phi(t^*)\vert \geq 1/2$ for some $t^*\geq 0$ then $\vert \phi(t)\vert \geq 1/2$ for all $t \geq t^*$;
\item if $\vert \phi(0)\vert >0$ and the solution is globally defined then
$$
\liminf_{t \to \infty} \lvert \phi(t) \rvert \geq 1/2;
$$
\item if $\vert \phi(0)\vert >0$ the solution never gets absorbed at the origin.
\end{itemize}
\end{thm}
\begin{proof}
First consider $|\phi(t)|^2=\phi_1(t)^2+\phi_2(t)^2$, where $\phi_1(t):=\Re(\phi)(t), \phi_2(t):=\Im(\phi)(t)$.
Since~\eqref{isde} is equivalent to the real two-dimensional system
\begin{equation}\nonumber
\left\lbrace
\begin{array}{l}
d \phi_1 = - (\phi_1^2 - \phi_2^2) dt- \phi_2 dW_t \\
d \phi_2 = - 2\phi_1 \phi_2 dt+ \phi_1 d W_t
\end{array}
\right. ,
\end{equation}
It\^o calculus implies
\begin{displaymath}
d|\phi|^2=(1-2 \phi_1)|\phi|^2 dt.
\end{displaymath}
This  differential together with the fact $-\vert \phi \vert \leq \phi_1 \leq \vert \phi \vert$ yields the string of inequalities
\begin{displaymath}
(1-2 \vert \phi \vert)\vert \phi \vert^2 \leq \frac{d}{dt}\vert \phi \vert^2 \leq (1+2\vert \phi \vert)\vert \phi \vert^2,
\end{displaymath}
which can be re-written as
\begin{displaymath}
\left(\dfrac{1}{2}- \vert \phi \vert \right)\vert \phi \vert \leq \frac{d}{dt}\vert \phi \vert \leq \left( \dfrac{1}{2}+\vert \phi \vert \right)
\vert \phi \vert.
\end{displaymath}
Dividing all three terms by $\lvert \phi \rvert ^2$, everything can be written in terms of $\gamma(t):=\lvert \phi(t)\rvert ^{-1}$
\begin{displaymath}
-\dfrac{\gamma(t)}{2}- 1 \leq \frac{d}{dt} \gamma(t) \leq -\dfrac{\gamma(t)}{2}+ 1,
\end{displaymath}
from where, after multiplying by $e^{t/2}$, we find
\begin{displaymath}
- e^{t/2} \leq \frac{d}{dt} \left( e^{t/2} \gamma(t)\right) \leq e^{t/2}.
\end{displaymath}
Integrating and undoing the change of variables we conclude
\begin{displaymath}
\dfrac{|\phi(t^*)| e^{(t-t^*)/2}}{1+2|\phi(t^*)|[e^{(t-t^*)/2}-1]} \leq \vert \phi(t) \vert \leq
\dfrac{|\phi(t^*)| e^{(t-t^*)/2}}{1-2|\phi(t^*)|[e^{(t-t^*)/2}-1]},
\end{displaymath}
for all $t \ge t^*$.\footnote{For the sake of clarity, let us emphasize that these computations are justified by the assumption on the initial condition
$\lvert \phi(0)\rvert > 0$ and the continuity of the paths of $\phi$, which guarantee their local validity that can be subsequently extended for
arbitrarily long intervals of time.}
Consequently:
\begin{itemize}
\item the solution exists until a possible blow-up time $T^* \geq 2 \log \left[1+\dfrac{1}{2|\phi(t^*)|}\right]$;
\item
$$
\lim_{t \to \infty} \dfrac{|\phi(t^*)| e^{(t-t^*)/2}}{1+2|\phi(t^*)|[e^{(t-t^*)/2}-1]} = \frac12;
$$
\item
$$
\dfrac{|\phi(t^*)| e^{(t-t^*)/2}}{1+2|\phi(t^*)|[e^{(t-t^*)/2}-1]}=\dfrac{1}{2+\dfrac{1-2|\phi(t^*)|}{|\phi(t^*)| e^{(t-t^*)/2}}};
$$
\item
$$
\dfrac{|\phi(t^*)| e^{(t-t^*)/2}}{1+2|\phi(t^*)|[e^{(t-t^*)/2}-1]}=0 \Longleftrightarrow |\phi(t^*)|=0;
$$
\end{itemize}
therefore the statement follows.
\end{proof}

\begin{cor}
Let $\phi(t)$ be a solution of \eqref{isde} such that $|\phi(0)|>0$; then there exists a $\bar{t} \in [0,\infty]$ such that $\phi(\bar{t}) \geq 1/2$.
Moreover if $\vert \phi\left(t^*\right)\vert > 1/2$ for some $0 \le t^* < \infty$ then $\vert \phi(t)\vert > 1/2$ for all $t^* \le t < \infty$.
\end{cor}

\begin{proof}
Follows immediately from the proof of Theorem~\ref{thdynamic1}.
\end{proof}

\begin{cor}\label{compsupp}
Let $\phi(t)$ be a solution of \eqref{isde} subject to $\phi(0)=x_0 \in \mathbb{R}$; with probability one it holds that:
$$
\vert \phi(t) \vert \leq
\dfrac{|x_0| e^{(t-t^*)/2}}{1-2|x_0|[e^{(t-t^*)/2}-1]}.
$$
In particular this implies that  solutions to the Fokker-Planck equation~\eqref{realfokker0}, with initial condition
$P(0, z_1,z_2)=\delta(z_1-x_0) \delta(z_2)$,
are compactly supported at least during the time interval $[0,\mathfrak{T})$, where the deterministic time
$$
\mathfrak{T} = 2 \log \left[1+\dfrac{1}{2|x_0|}\right].
$$
\end{cor}

\begin{proof}
Follows immediately from the proof of Theorem~\ref{thdynamic1}.
\end{proof}

\begin{remark}
Note that this corollary assures the compact support and the finiteness of moments requirements in the statements of Theorems~\ref{connection} and~\ref{connection2} locally in time. This is due to the  degenerated character of the diffusion and the specific structure of the coefficients.
%(there is just one noise in two dimensions).
Although these theorems build the Cauchy representation by averaging over $x_0$, the compact support
of the initial distribution guarantees their local in time validity. Moreover, since the initial distribution is supported at the origin
in the case of Theorem~\ref{connection2}, the validity of this theorem is actually global in time
(in the sense that it is valid for any finite, but arbitrary, lapse of time).
Note, however, that this initial distribution involves derivatives of the Dirac delta, so the solution of the Fokker-Planck equation has to be defined in
a neighborhood of $x_0=0$, what implies that we cannot rely only on the trivial solution of the SDE;
the global in time validity follows from the fact that this neighborhood can be arbitrarily small.
\end{remark}

We note that the second property in the statement of Theorem~\ref{thdynamic1} was observed in numerical simulations~\cite{munoz1} and
herein we offer, for the first time to the best of our knowledge, a mathematical proof of it.
Moreover, it shows that the probability density of this diffusion might identically vanish in a certain disc for all times.
Again, this is a consequence of the degenerated character of the diffusion and the particular structure of the coefficients.
%This is again a consequence of the degeneracy of the diffusion and connects with the compact support requirement in the statement of Theorem~\ref{connection} (and indirectly with the statement of Theorem~\ref{connection2}).
Note that we have not proven the global existence of solution and, in fact, a finite time blow-up is in principle possible
due to the quadratic nonlinearity. If global existence held, then it would be interesting to determine whether or not
the compact support property is global in time too. But so far both questions remain open.
Next we state the counterpart of Theorem~\ref{thdynamic1} in terms of standard deviations.

\begin{thm}
Let $\phi(t)$ be a solution to~\eqref{isde} such that $\vert \phi(0)\vert >0$; then
\begin{itemize}
\item If $\mathbb{E}[\vert \phi(t^*)\vert^{-2}] < \infty$
for some $t^* \geq 0$ then $\liminf_{t \to \infty} \mathbb{E}[\vert \phi(t)\vert^2]^{1/2} \geq 1/\sqrt{2}$.
\item Moreover if $\mathbb{E}[\vert \phi(t^*)\vert^{-2}]^{-1/2} \geq 1/2$
for some $t^* \geq 0$ then $\mathbb{E}[\vert \phi(t)\vert^2]^{1/2} \geq 1/2$ for all $t \geq t^*$.
\end{itemize}
\end{thm}

\begin{proof}
Since $\vert \phi(0)\vert >0$ then $\vert \phi(t)\vert >0$ for all times by Theorem~\ref{thdynamic1}.
Consequently we can change variables $\xi=1/ \phi$ to obtain
\begin{displaymath}
d\xi=(1-\xi)dt-i\xi dW_t,
\end{displaymath}
which is a linear stochastic differential equation in the complex plane and therefore globally well-posed,
as well as the change of variables.
Therefore the expectation of $\xi$ obeys the ordinary differential equation
\begin{displaymath}
\dfrac{d}{dt}\mathbb{E}[\xi]=1-\mathbb{E}[\xi],
\end{displaymath}
which can be solved to yield
\begin{eqnarray}\label{expected_eq}
\mathbb{E}[\xi(t)] &=& 1+(\mathbb{E}[\xi(t^*)]-1)e^{-(t-t^*)} \\ \nonumber
&=& 1+\left(\mathbb{E}\left[\phi(t^*)^{-1}\right]-1\right)e^{-(t-t^*)}.
\end{eqnarray}
Now consider $f(\xi_1(t),\xi_2(t)):=\frac12 |\xi(t)|^2=\frac12(\xi_1(t)^2+\xi_2(t)^2)$, where $\xi_1(t):=\Re(\xi(t)), \xi_2:=\Im(\xi)(t)$, which
obeys the random differential equation
$$
df=(\xi_1-f)dt,
$$
and therefore using \eqref{expected_eq} its expectation fulfills
$$
\dfrac{d}{dt}\mathbb{E}[f]=1+(\mathbb{E}[\xi_1(t^*)]-1)e^{-(t-t^*)}-\mathbb{E}[f],
$$
which solution reads
\begin{eqnarray}\nonumber
\mathbb{E}[f(t)] &=& 1-e^{-(t-t^*)}+(\mathbb{E}[\xi_1(t^*)]-1)(t-t^*)e^{-(t-t^*)}+e^{-(t-t^*)}\mathbb{E}[f(t^*)] \\ \nonumber
&=& 1-e^{-(t-t^*)}+\left(\mathbb{E}\left[\frac{\phi_1(t^*)}{|\phi(t^*)|^2}\right]-1\right)(t-t^*)e^{-(t-t^*)}+e^{-(t-t^*)}\mathbb{E}[f(t^*)] \\ \nonumber
&\le& 1 + e^{-(t-t^*)} + \left(\mathbb{E}\left[\frac{\phi_1(t^*)}{|\phi(t^*)|^2}\right]-1\right)(t-t^*)e^{-(t-t^*)} \\ \nonumber
&\le& 1 + (1+t-t^*)e^{-(t-t^*)}  \le  2.
\end{eqnarray}
Here we used that $\mathbb{E}[\vert \phi(t^*)\vert^{-2}]^{-1/2} \geq 1/2 \Rightarrow \mathbb{E}[f(t^*)] \le 2$ in the first inequality and
\begin{eqnarray}\nonumber
\mathbb{E}\left[\frac{\phi_1(t^*)}{|\phi(t^*)|^2}\right] &\le& \mathbb{E}\left[\frac{1}{|\phi(t^*)|}\right] \le  \mathbb{E}\left[\frac{1}{|\phi(t^*)|^2}\right]^{1/2} \le 2,
\end{eqnarray}
by Jensen inequality, in the second. Now, using H{\"o}lder inequality
\begin{eqnarray}\nonumber
1 &=& \mathbb{E}[1] = \mathbb{E}\left[\dfrac{\left\lvert \phi(t) \right\rvert}{\left\lvert \phi(t) \right\rvert}\right] \le \mathbb{E}[|\phi(t)|^2]^{1/2} \, \mathbb{E}[|\phi(t)|^{-2}]^{1/2} \\ \nonumber
&=& \sqrt{2} \, \mathbb{E}[|\phi(t)|^2]^{1/2} \, \mathbb{E}[f(t)]^{1/2} \le 2 \, \mathbb{E}[|\phi(t)|^2]^{1/2}.
\end{eqnarray}
For the asymptotic behavior, take the long time limit in this string of inequalities to find
\begin{eqnarray}\nonumber
1 &\le& \sqrt{2} \, \liminf_{t \to \infty} \left\{ \mathbb{E}[|\phi(t)|^2]^{1/2} \, \mathbb{E}[f(t)]^{1/2} \right\} \\ \nonumber
&=& \sqrt{2} \, \liminf_{t \to \infty} \, \mathbb{E}[|\phi(t)|^2]^{1/2} \times \\ \nonumber
& & \times \lim_{t \to \infty} \left\{
1-e^{-(t-t^*)}+\left(\mathbb{E}\left[\frac{\phi_1(t^*)}{|\phi(t^*)|^2}\right]-1\right)(t-t^*)e^{-(t-t^*)}+e^{-(t-t^*)}\mathbb{E}[f(t^*)]
\right\}^{1/2} \\ \nonumber
&=& \sqrt{2} \, \liminf_{t \to \infty} \, \mathbb{E}[|\phi(t)|^2]^{1/2}.
\end{eqnarray}
\end{proof}

\begin{remark}
Note that by H{\"o}lder inequality
\begin{eqnarray}\nonumber
1 &=& \mathbb{E}[1] = \mathbb{E}\left[\dfrac{\left\lvert \phi(t) \right\rvert}{\left\lvert \phi(t) \right\rvert}\right] \leq  \mathbb{E}[|\phi(t)|^2]^{1/2} \, \mathbb{E}[|\phi(t)|^{-2}]^{1/2} \\ \nonumber
&\Rightarrow& \mathbb{E}[|\phi(t)|^{-2}]^{-1/2} \le \mathbb{E}[|\phi(t)|^2]^{1/2}.
\end{eqnarray}
Therefore $\mathbb{E}[\vert \phi(t^*)\vert^{-2}]^{-1/2} \geq 1/2$ is a condition stronger than $\mathbb{E}[\vert \phi(t^*)\vert^{2}]^{1/2} \geq 1/2$.
\end{remark}

\section{Conclusions}
\label{conclusions}

In this work we have considered the connection between the PDE
\begin{equation}\nonumber
\dfrac{\partial \Psi}{\partial t}=\left(\dfrac{\partial}{\partial \phi}-\dfrac{1}{2}\dfrac{\partial^2}{\partial \phi^2}\right)
\left( \phi^2 \Psi \right),
\end{equation}
and the SDE
\begin{equation}\nonumber
d \phi = - \phi^2 dt + i \, \phi \, dW_t,
\end{equation}
that appears in the physics literature to study chemical kinetics modeled by Markov chains. This relation, which we have termed the imaginary
It\^o interpretation
of the PDE, has been accepted in the physical literature for decades but was also put into question in some works. From a puristic viewpoint, of course,
one cannot claim that a parabolic PDE with a negative diffusion is a Fokker-Planck equation. Perhaps more importantly, if one regards for instance the solutions
in appendix~\ref{expformula} for $t=0$ (that is, the initial conditions), one finds distributions rather than measures, what means that the initial condition for the SDE does not exist, at least as a random variable; obviously this suggests a very difficult, if not impossible, interpretation of the SDE.

Keeping these facts in mind one is tempted to claim that the imaginary It\^o interpretation is nothing but a formal step that cannot be justified. However,
the successes in the application of this theory (or different facets of it) \cite{cardy2, cardy5, cardy6, dfhk,howard,lee, munoz1,peliti1},
although they are all based on formal computations, point to the opposite direction.
This has been the motivation to build our connection between the two theories in section~\ref{notunreal}, which in principle should be valid for SDEs that
have as solution well-defined diffusion processes on the complex plane. In our particular case, a missing step in our proofs is the global existence of
the solution to the SDE. Although we have partially analyzed its dynamics in section~\ref{dynamics}, we have not found such an argument that would guarantee
the existence of a diffusion on the complex plane for all times. Nevertheless, the local existence in time guarantees that we can extend the Cauchy
representation in Theorem~\ref{connection2} for any finite, but arbitrary, lapse of time, since the initial distribution is supported
at the origin (despite its dependence on derivatives of the Dirac delta, what makes nontrivial solutions of the SDE to come into play).
This, in our view, establishes a clear connection between the imaginary It\^o interpretation and the forward Kolmogorov equation when the initial condition of the latter is compactly supported. Otherwise, the connection is established for short times via Theorem~\ref{connection}.

It is also important to try to see how our present results could match with recent criticisms to the imaginary It\^o interpretation. In~\cite{wiese}
one finds a claim that points to the validity of the imaginary It\^o interpretation at short times and its failure at long or even intermediate
times. This could perhaps be related to the singularization of the probability amplitude $\Psi$: if the initial condition of the Markov chain is
Poissonian then the probability amplitude will be a probability measure initially too; however as the time evolves it will become singular
(i.e. a distribution rather than a measure). Of course this is just a conjecture and further analysis would be necessary in order to assure this.
In this respect, the lapse of existence guaranteed in section~\ref{dynamics} might be related to this short time validity. Anyway, this lapse of
existence was not crucial in our analysis as we encoded the initial condition in a different way, via the initial
distribution (so what is really crucial in our case is the size of the support of this initial distribution); this suggests
in turn that there may be different notions of imaginary It\^o interpretation present in the literature.
In~\cite{benitez} the authors put into question the validity of the imaginary It\^o interpretation through a formal path integral analysis:
they conclude this by means of the identification of a path integral that is ill-posed. However, a parabolic PDE provided with a negative diffusion
is ill-posed, at least in the sense of Hadamard (Lemma 1.19, \cite{renar}), but nevertheless it could be well-posed in certain
distributional spaces~\cite{aragaosilva} or under additional conditions~\cite{miran}.
We do not know whether or not such extensions in the notion of solution can be carried out in the case of the
path integral too. Independently of this, the possibility that different notions of imaginary It\^o interpretation are being considered should
not be immediately disregarded in this case either.

Of course, another possible criticism to the theory of the imaginary It\^o interpretation is its potential utility. Mapping a continuous time Markov
chain into a PDE posed in a space of distributions looks like making a difficult problem an extremely difficult one instead. However, the previous successes
referred to above in the use of this framework suggest the interest that exploring the stochastic analytical side of it may have.
Correspondingly, it may also be interesting to study distributional PDEs or complex plane diffusions by means of Markov chains; indeed, the study of It\^o
diffusions by means of Markov chains is known to be simplifying and has been explored within the framework of Malliavin calculus~\cite{kohatsu}.

Finally, we wonder whether the imaginary It\^o interpretation is part of a bigger theory that links diffusions with PDEs. Apart from the classical
diffusion theory that links Fokker-Planck equations with SDEs~\cite{oksendal1}, one finds different theories that approach higher order and
fractional order PDEs with stochastic processes and pseudoprocesses, see for instance~\cite{vallois,debbi,funaki,hochberg,hochberg1,krylov,ovidio}.
Since we can regard the imaginary
It\^o interpretation as a link between singular second order PDEs and diffusions on the complex plane, there arises a natural question about the
extendability of this theory to the singular higher order and fractional order cases, and even about the existence of a general theory that
comprises all these connections as particular cases of a more general relation.\\

%\vspace{1cm}

{\bf Acknowledgments.} {This work has been partially supported by a NRP Early Career Research Exchanges grant,
by the ICMAT-Severo Ochoa project, and by project PGC2018-097704-B-I00 of the Ministerio de Ciencia, Innovaci\'on y Universidades (Spain).}\\
%\pagebreak

\begin{appendices}

\section{Van Kampen System Size Expansion}\label{sectvan}

In this appendix we formally derive a mean-field macroscopic limit of reaction~\eqref{mainreact}
along the lines of the Van Kampen system size expansion~\cite{vankampen}.
Consider the forward Kolmogorov equation
\begin{displaymath}
\frac{d P_n(t)}{dt} = \frac{\lambda}{2} \left[(n+2)(n+1)P_{n+2}(t)-n(n-1)P_n(t) \right],
\end{displaymath}
and assume the existence of a regular enough function $f:\mathbb{R}_+ \times \mathbb{R}_+ \longrightarrow \mathbb{R}$,
with $\mathbb{R}_+ := [0,\infty)$, such that $f(\tau, x)=\lim_{\delta \rightarrow 0} P_{n\delta}(t \delta^{-1})$, $n \delta \to x$,
and $t \delta^{-1} \to \tau$. Keeping these ideas in mind, we can derive an approximate equation for $f$ via a Taylor expansion:
\begin{eqnarray}\nonumber
\delta^{-1}\dfrac{\partial f}{\partial \tau} (n\delta, t\delta^{-1}) &=&
\dfrac{\lambda}{2} \left[(n+2)(n+1)f(n\delta +2\delta , t\delta^{-1})-n(n-1)f(n\delta, t\delta^{-1})\right] \\ \nonumber
&=& \lambda \left[(2n+1)f(n\delta, t\delta^{-1})+\delta (n+2)(n+1)\dfrac{\partial f}{\partial x}(n\delta, t\delta^{-1}) \right.
\\ \nonumber && \left. + \, \delta^2 (n+2)(n+1)\dfrac{\partial^2 f}{\partial x^2}(n\delta, t\delta^{-1})+\dots \right];
\end{eqnarray}
now multiplying this equation by $\delta$ and formally taking the limit $\delta \to 0$ leads to
\begin{eqnarray}\nonumber
\dfrac{\partial f}{\partial \tau} (x, \tau) &=& \lambda \left[2x f(x, \tau)+x^2 \dfrac{\partial f}{\partial x}(x, \tau)\right] \\ \label{fpmf}
&=& \lambda \, \dfrac{\partial}{\partial x}\left( x^2 f \right).
\end{eqnarray}
Let $\Phi(x)$ be an infinitely differentiable function with compact support. We define a function $f$ to be a distributional solution
to equation~\eqref{fpmf} if, for every $\Phi$, the equality
\begin{displaymath}
\dfrac{d}{d \tau} \int_{\mathbb{R}_+} f \, \Phi \, dx = -\lambda \int_{\mathbb{R}_+} f \, x^2 \, \dfrac{\partial \Phi}{\partial x} \, dx
\end{displaymath}
holds. More generally, define a linear distribution $D_\tau:\mathbb{R}_{+} \times C^{\infty}_{0} \longrightarrow \mathbb{R}$
to be a solution of this equation if the following equality holds for every test function $\Phi$:
\begin{displaymath}
\dfrac{d}{d\tau}D_\tau[\Phi]=-\lambda D_\tau \left[x^2 \dfrac{\partial \Phi}{\partial x}\right].
\end{displaymath}
In particular, if we want the solution to be expressed in terms of a Dirac delta of the form $D_\tau=\delta_{\phi(\tau)}$, then:
\begin{displaymath}
\phi'(\tau) \dfrac{\partial \Phi}{\partial x}[\phi(\tau)]=-\lambda \phi^2(\tau) \dfrac{\partial \Phi}{\partial x}[\phi(\tau)]
\end{displaymath}
for every test function, what implies
\begin{displaymath}
\dfrac{d\phi(\tau)}{dt}=-\lambda \phi^2(\tau),
\end{displaymath}
which could be thought of as a mean-field approximation to equation~\eqref{isde}.

\section{Equations for the General Reaction}\label{sectgeneral}

In this appendix we derive the equations for the generating function and the amplitude that correspond to the annihilation-creation Markov process
\begin{displaymath}
jA\stackrel{\lambda}{\longrightarrow} \ell A,
\end{displaymath}
with $j,\ell \in \mathbb{N}\cup \{0\}$, obviously $j \neq \ell$, and $\lambda > 0$.
The forward Kolmogorov equation corresponding to this process reads \cite{mc}
\begin{displaymath}\label{appa:master}
\dfrac{1}{\lambda}\dfrac{d P_n(t)}{dt}=
\left(
\begin{array}{c}
n-\ell+j \\ j
\end{array}
\right) P_{n-\ell+j}(t)-
\left(
\begin{array}{c}
n \\ j
\end{array}
\right) P_{n}(t),
\end{displaymath}
for all $n\in \mathbb{N}\cup \{0\}$ and with the understanding that $P_m(t) \equiv 0$ whenever $m < 0$.

Using the generating function representation of this system,
\begin{displaymath}
G(t,x):=\sum_{n=0}^{\infty}P_n(t)x^n,
\end{displaymath}
leads to
\begin{eqnarray}\nonumber
\dfrac{1}{\lambda}\dfrac{\partial G}{\partial t} &=& \sum_{n=0}^{\infty} \frac{dP_n}{dt} x^n
\\ \nonumber
&=& \dfrac{1}{j!}\left[ \sum_{n=0}^{\infty} (n-\ell+j)\cdots(n-l+1)P_{n-\ell+j}x^n-\sum_{n=0}^{\infty} n\cdots(n-j+1)P_{n}x^n\right],
\end{eqnarray}
whenever $j \neq 0$, and thus
\begin{eqnarray*}
\dfrac{1}{\lambda}\dfrac{\partial G}{\partial t} & = & \dfrac{1}{j!}\left[ \sum_{n=j-\ell}^{\infty} n\cdots(n-j+1)P_{n}x^{n+\ell-j}-\sum_{n=0}^{\infty} n\cdots(n-j+1)P_{n}x^n\right] \\ & = & \dfrac{1}{j!}\left[ \sum_{n=0}^{\infty} n\cdots(n-j+1)P_{n}x^{n+\ell-j}-\sum_{n=0}^{\infty} n\cdots(n-j+1)P_{n}x^n\right]\\ & = & \dfrac{1}{j!}\left[ x^\ell \dfrac{\partial^j}{\partial x^j}G-x^j\dfrac{\partial^j}{\partial x^j}G \right];
\end{eqnarray*}
so we can conclude
\begin{equation}\label{eqgen}
\dfrac{\partial G}{\partial t}=\dfrac{\lambda}{j!}(x^\ell-x^j)\dfrac{\partial^j G}{\partial x^j}.
\end{equation}
An analogous computation shows that the case $j=0$ is still described by equation~\eqref{eqgen}.

The amplitude $\Psi(t,\phi)$ is related to the generating function via the coherent transform $\Psi(t,\phi)$:
\begin{displaymath}
G(t,x)=\int_0^{\infty}  \Psi(t,\phi) e^{\phi(x-1)} d\phi.
\end{displaymath}
Using the binomial theorem we find
\begin{displaymath}
x^\ell-x^j=\sum_{n=0}^{\ell \vee j} \left[ \left(
\begin{array}{c}
l \\ n
\end{array}
\right)-\left(
\begin{array}{c}
j \\ n
\end{array}
\right)\right] (x-1)^n,
\end{displaymath}
where we have used the convention $\left(
\begin{array}{c}
\ell \\ n
\end{array}
\right)=0$ if $n> \ell$ and $\ell \vee j := \max\{\ell,j\}$.
These two observations together with equation~\eqref{eqgen} and integration by parts lead to the result:
\begin{displaymath}
\dfrac{\partial \Psi}{\partial t}=\dfrac{\lambda}{j!}\sum_{n=0}^{j} (-1)^n\left[ \left(
\begin{array}{c}
\ell \\ n
\end{array}
\right)-\left(
\begin{array}{c}
j \\ n
\end{array}
\right)\right]\dfrac{\partial^n (\phi^j \Psi)}{\partial \phi^n},
\end{displaymath}
where we have assumed $j>\ell$, which is a necessary condition in order to eliminate the boundary terms generated by integrating by parts.

This last formula implies that the only two reactions that lead to second order operators are:
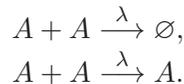
\begin{displaymath}
\begin{array}{c}
A+A \stackrel{\lambda}{\longrightarrow} \varnothing,  \\
A+A \stackrel{\lambda}{\longrightarrow} A.
\end{array}
\end{displaymath}

\section{An Explicit Distributional Solution}\label{expformula}

Based on the explicit solution to equation~\eqref{mark11} found in~\cite{Mcq2},
it is possible to derive an explicit solution to the formal Fokker-Planck equation~\eqref{bad1}:
\begin{eqnarray}\nonumber
\Psi(t,\phi) &=& \delta(\phi) + \sum_{k=1}^{k_0} \frac{(2k_0)! (k_0+k)!}{2^{-2k}(2k_0 + 2k)! (k_0-k)!} e^{-k(2k-1)\lambda t} \\ \nonumber &&
\times \left[ \sum_{j=0}^{2k} 2^{-j} \binom{2k}{j} \binom{-2k-1}{j} \delta^{(j)}(\phi)
-\sum_{j=0}^{2k-2} 2^{-j} \binom{2k-2}{j} \binom{-2j+3}{j} \delta^{(j)}(\phi)\right],
\end{eqnarray}
when initially we have exactly $2 k_0$ particles, and where $\delta^{(j)}$ denotes the $j-$th derivative of the Dirac delta and
$$
\binom{-n}{m}:=(-1)^m \binom{n+m-1}{m}
$$
for any $n,m \in \mathbb{Z}^+$. As an example, we consider this formula in the case of having initially exactly two particles:
\begin{equation}\nonumber
\Psi(t,\phi)=\delta_0+e^{-\lambda t}\left\lbrace \delta_0^{(2)}-2\delta_0^{(1)} \right\rbrace.
\end{equation}
Note that this formula shows that solutions to the ill-posed Fokker-Planck equation are in
general distributions rather than measures.

\section{More on Real Noise}\label{morern}

In this Appendix we consider the set of reactions
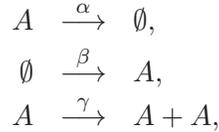
\begin{eqnarray} \nonumber
A &\stackrel{\alpha}{\longrightarrow}& \emptyset, \\ \nonumber
\emptyset &\stackrel{\beta}{\longrightarrow}& A, \\ \nonumber
A &\stackrel{\gamma}{\longrightarrow}& A + A,
\end{eqnarray}
which can be described via the forward Kolmogorov equation
$$
\frac{d P_n}{dt}= \gamma [(n-1)P_{n-1} -n P_n] + \beta(P_{n-1} - P_n) + \alpha [(n+1)P_{n+1} - nP_n].
$$
For our current purposes we make the choice $\alpha=\gamma=2\beta$; then we find the equation
$$
\partial_t G = \beta(x-1)(1 - 2 \partial_x + 2 x \partial_x)G,
$$
to be solved for the generating function $G$. Its solution reads
\begin{equation}\label{genmc}
G(t,x)= \frac{1}{\sqrt{1-2 \beta t (x-1)}} \,\, G_0 \left( \frac{x-2 \beta t (x-1)}{1-2 \beta t (x-1)} \right),
\end{equation}
where we have used the boundary condition $G(t,1) =1$. For the amplitude $\Psi$ we \emph{formally} find
$$
\partial_t \Psi = -\beta \partial_\phi \Psi + 2 \beta \partial_\phi^2 (\phi \Psi);
$$
note that we have neglected a Dirac delta that arises as a boundary term upon integration by parts in the derivation of this equation
from the coherent transform of the generating function (a problem that does not arise in section~\ref{rnoise}).
Note also that this PDE is the \emph{actual} Fokker-Planck equation that corresponds to the SDE
$$
d \phi = \beta \, dt + 2 \sqrt{\beta \phi} \, dW_t.
$$
The explicit solution of this equation is
$$
\phi(t)= \left( \sqrt{\phi_0} + \sqrt{\beta} \, W_t \right)^2,
$$
that in turn yields
\begin{eqnarray} \nonumber
\Psi(t,\phi)&= &\frac{1}{2\sqrt{8 \pi \beta t \phi}} \int_0^\infty \exp \left[ -\frac{ \left(\sqrt{\phi} -  \sqrt{\phi_0}\right)^2}{2 \beta t} \right]
\Psi_0(\phi_0) \, d\phi_0+ \\ \nonumber & &+\frac{1}{2\sqrt{8 \pi \beta t \phi}} \int_0^\infty \exp \left[ -\frac{ \left(\sqrt{\phi} + \sqrt{\phi_0}\right)^2}{2 \beta t} \right]
\Psi_0(\phi_0) \, d\phi_0 \\ \nonumber &=&\frac{1}{\sqrt{8 \pi \beta t \phi}} \int_0^\infty \exp \left[ -\frac{ \phi +  \phi_0}{2 \beta t} \right] \cosh \left[ \dfrac{\sqrt{\phi \phi_0}}{\beta t} \right]
\Psi_0(\phi_0) \, d\phi_0.
\end{eqnarray}
This result gives rise to the generating function
\begin{eqnarray}\nonumber
G(t,x) &=& \frac{1}{\sqrt{1-2 \beta t (x-1)}} \int_0^\infty \exp \left[ \frac{\phi_0 (x-1)}{1-2 \beta t (x-1)} \right]
\Psi_0(\phi_0) \, d \phi_0 \\ \nonumber
&=& \frac{1}{\sqrt{1-2 \beta t (x-1)}} \,\, G_0 \left( \frac{x-2 \beta t (x-1)}{1-2 \beta t (x-1)} \right),
\end{eqnarray}
in perfect agreement with~\eqref{genmc}. Note that this last result makes sense even if $\Psi_0$ is not a probability
measure.

\section{Preservation of It\^o Formula}\label{sectvs}

Consider a modification of our It\^o diffusion in the complex plane, i.e.
\begin{equation}\nonumber
d \phi = - \phi^2 dt + i \, \phi \, dW_t,
\end{equation}
given by
\begin{equation}\nonumber
d \varphi = - \varphi^2 dt + i \, \varphi \, d\mathcal{W}_t,
\end{equation}
where $\mathcal{W}_t=W^1_t +i W^2_t$, and $W^1_t$ and $W^2_t$ are independent Brownian motions; in other words $\mathcal{W}_t$ is a complex Brownian motion.
This second model is invalid from our viewpoint as it does not generate the right moments as the first one does, see Lemma~\ref{theo:moments};
however it rises an interesting question. It is well known that for any holomorphic function $f(\varphi)$ the usual chain rule
\begin{displaymath}
df(\varphi)=-\varphi^2 \dfrac{\partial f(\varphi)}{\partial \varphi} dt + i \varphi \dfrac{\partial f(\varphi)}{\partial \varphi} d\mathcal{W}_t
\end{displaymath}
holds (\cite{peres}, section 7.2), but what is the right calculus for $f(\phi)$?
Herein we show that, in the case in which we have purely imaginary noise, the usual It\^o stochastic calculus
\begin{equation}\label{sde_main2}
df(\phi)=-\phi^2 \left[  \dfrac{\partial f(\phi)}{\partial \phi}+\dfrac{1}{2}\dfrac{\partial^2f(\phi)}{\partial \phi^2}\right] dt
+i\phi \dfrac{\partial f(\phi)}{\partial \phi}dW_t
\end{equation}
still holds.
Now it is convenient to use the decomposition $\phi=\phi_1 + i \phi_2$ to transform our diffusion in the complex plane into a two-dimensional It\^o diffusion
\begin{equation}\nonumber
\left(
\begin{array}{c}
d\phi_1 \\ d\phi_2
\end{array}
\right) =-\left(
\begin{array}{c}
\phi_1^2-\phi_2^2 \\ 2\phi_1 \phi_2
\end{array}
\right)dt + \left(
\begin{array}{cr}
-\phi_2 & 0 \\ \phi_1 & 0
\end{array}
\right) \left(
\begin{array}{c}
dW^1_t \\ dW^2_t
\end{array}
\right)
\end{equation}
in order to prove the following precise result.

\begin{thm}
Let $f(\phi)=f_1(\phi)+if_2(\phi)$ be a holomorphic function; then equation~\eqref{sde_main2} holds, where
\begin{displaymath}
\left\lbrace
\begin{array}{l}
\dfrac{\partial }{\partial \phi}=\dfrac{1}{2}\left(\dfrac{\partial}{\partial \phi_1}-i\dfrac{\partial}{\partial \phi_2}\right) \\ [0.4cm] \dfrac{\partial }{\partial \overline{\phi}}=\dfrac{1}{2}\left(\dfrac{\partial}{\partial \phi_1}+i\dfrac{\partial}{\partial \phi_2}\right)
\end{array}
\right..
\end{displaymath}
\end{thm}

\begin{proof}
	The fact that $f$ is holomorphic, which implies $\dfrac{\partial f}{\partial \overline{\phi}} = 0$,
along with the definitions of $\dfrac{\partial f}{\partial \phi}$ and $\dfrac{\partial f}{\partial \overline{\phi}}$, yield the following identities
	\begin{eqnarray*}
		\dfrac{\partial f}{\partial \phi}   & = & \dfrac{\partial f}{\partial \phi_1}  =  -i\dfrac{\partial f}{\partial \phi_2}, \\
		\dfrac{\partial^2 f}{\partial \phi^2} & = &  \dfrac{\partial^2 f}{\partial \phi_1^2} =  -\dfrac{\partial^2 f}{\partial \phi_2^2}  =  -i\dfrac{\partial^2 f}{\partial \phi_1 \partial \phi_2}.
	\end{eqnarray*}
	On the other hand, the two dimensional It\^o rule yields
	\begin{eqnarray*}
	df & = & \left[-(\phi_1^2 - \phi_2^2)\dfrac{\partial f}{\partial \phi_1} - 2\phi_1\phi_2\dfrac{\partial f}{\partial \phi_2} + \dfrac{1}{2}\phi_2^2\dfrac{\partial^2f}{\partial \phi_1^2} +
	\dfrac{1}{2}\phi_1^2\dfrac{\partial^2f}{\partial \phi_2^2} - \phi_1\phi_2\dfrac{\partial^2 f}{\partial \phi_1\phi_2}  \right] dt  \\ & & + \left[ \phi_1 \dfrac{\partial f}{\partial \phi_2} - \phi_2 \dfrac{\partial f}{\partial \phi_1} \right] dW_t.
	\end{eqnarray*}
	Now, by substituting the previous expressions, we find
	\begin{eqnarray*}
	df & = & \left[-(\phi_1^2 - \phi_2^2)\dfrac{\partial f}{\partial \phi} - i 2\phi_1\phi_2\dfrac{\partial f}{\partial \phi} + \dfrac{1}{2}\phi_2^2\dfrac{\partial^2f}{\partial \phi^2} -
	\dfrac{1}{2}\phi_1^2\dfrac{\partial^2f}{\partial \phi^2} - i \phi_1\phi_2\dfrac{\partial^2 f}{\partial \phi^2}  \right] dt  \\ & & + \left[ i\phi_1\dfrac{\partial f}{\partial \phi} - \phi_2 \dfrac{\partial f}{\partial \phi} \right] dW_t
	%\\& = & \left[- \phi^2 \dfrac{\partial f}{\partial \phi} - \dfrac{1}{2} \phi^2 \dfrac{\partial^2f}{\partial \phi^2} \right] dt  +  i \phi \dfrac{\partial f}{\partial \phi} dW_t
	\\ &=& - \phi^2 \left[ \dfrac{\partial f}{\partial \phi} + \dfrac{1}{2} \dfrac{\partial^2f}{\partial \phi^2} \right] dt  +  i \phi \dfrac{\partial f}{\partial \phi} dW_t.
	\end{eqnarray*}
\end{proof}

\end{appendices}

\vskip5mm
\noindent
{\footnotesize
	\'Alvaro Correales\par\noindent
	Departamento de Matem\'aticas\par\noindent
	Universidad Aut\'onoma de Madrid\par\noindent
	{\tt alvaro.correales@estudiante.uam.es}\par\noindent
	Carlos Escudero\par\noindent
	Departamento de Matem\'aticas Fundamentales\par\noindent
	Universidad Nacional de Educaci\'on a Distancia\par\noindent
	{\tt cescudero@mat.uned.es}\par\vskip1mm\noindent
	Mariya Ptashnyk\par\noindent
	Department of Mathematics\par\noindent
	Heriot-Watt University\par\noindent
	{\tt m.ptashnyk@hw.ac.uk}\par\vskip1mm\noindent
}
\end{document}